\pdfoutput=1
\documentclass[11pt, a4paper, twoside,leqno]{amsart}
\usepackage[centering, totalwidth = 380pt, totalheight = 600pt]{geometry}
\usepackage{amssymb, amsmath, amsthm}
\usepackage{microtype, stmaryrd, url, lmodern, thmtools, thm-restate,xspace}
\usepackage[shortlabels]{enumitem}
\usepackage[latin1]{inputenc}
\usepackage{color}
\definecolor{darkgreen}{rgb}{0,0.45,0}
\usepackage[colorlinks,citecolor=darkgreen,linkcolor=darkgreen]{hyperref}
\usepackage[british]{babel}

\DeclareMathAlphabet{\mathbf}{OT1}{cmr}{b}{n}

\usepackage[arrow, matrix, tips, curve, graph, rotate]{xy}
\SelectTips{cm}{10}

\makeatletter
\def\matrixobject@{
  \edef \next@{={\DirectionfromtheDirection@ }}%
  \expandafter \toks@ \next@ \plainxy@
  \let\xy@@ix@=\xyq@@toksix@
  \xyFN@ \OBJECT@}
\let\xy@entry@@norm=\entry@@norm
\def\entry@@norm@patched{%
  \let\object@=\matrixobject@
  \xy@entry@@norm }
\AtBeginDocument{\let\entry@@norm\entry@@norm@patched}
\makeatother

\newcommand{\twosim}[2][0.5]{\ar@{}[#2] \save ?(#1)*{\simeq}\restore}
\newcommand{\twocong}[2][0.5]{\ar@{}[#2] \save ?(#1)*{\cong}\restore}
\newcommand{\twoeq}[2][0.5]{\ar@{}[#2] \save ?(#1)*{=}\restore}
\newcommand{\rtwocell}[3][0.5]{\ar@{}[#2] \ar@{=>}?(#1)+/l 0.2cm/;?(#1)+/r 0.2cm/^{#3}}
\newcommand{\rtwocello}[3][0.5]{\ar@{}[#2] \ar@{=>}?(#1)+/l 0.2cm/;?(#1)+/r 0.2cm/_{#3}}
\newcommand{\ltwocell}[3][0.5]{\ar@{}[#2] \ar@{=>}?(#1)+/r 0.2cm/;?(#1)+/l 0.2cm/^{#3}}
\newcommand{\ltwocello}[3][0.5]{\ar@{}[#2] \ar@{=>}?(#1)+/r 0.2cm/;?(#1)+/l 0.2cm/_{#3}}
\newcommand{\dtwocell}[3][0.5]{\ar@{}[#2] \ar@{=>}?(#1)+/u  0.2cm/;?(#1)+/d 0.2cm/^{#3}}
\newcommand{\dltwocell}[3][0.5]{\ar@{}[#2] \ar@{=>}?(#1)+/ur  0.2cm/;?(#1)+/dl 0.2cm/^{#3}}
\newcommand{\urtwocell}[3][0.5]{\ar@{}[#2] \ar@{=>}?(#1)+/dl  0.2cm/;?(#1)+/ur 0.2cm/^{#3}}
\newcommand{\drtwocell}[3][0.5]{\ar@{}[#2] \ar@{=>}?(#1)+/ul  0.2cm/;?(#1)+/dr 0.2cm/^{#3}}
\newcommand{\dthreecell}[3][0.5]{\ar@{}[#2] \ar@3{->}?(#1)+/u  0.2cm/;?(#1)+/d 0.2cm/^{#3}}
\newcommand{\utwocell}[3][0.5]{\ar@{}[#2] \ar@{=>}?(#1)+/d 0.2cm/;?(#1)+/u 0.2cm/_{#3}}
\newcommand{\dtwocelltarg}[3][0.5]{\ar@{}#2 \ar@{=>}?(#1)+/u  0.2cm/;?(#1)+/d 0.2cm/^{#3}}
\newcommand{\utwocelltarg}[3][0.5]{\ar@{}#2 \ar@{=>}?(#1)+/d  0.2cm/;?(#1)+/u 0.2cm/_{#3}}

\newcommand{\xdrtwocell}[3][0.5]{\ar@{}[#2] \ar@{=>}?(#1)+/ul  0.14cm/;?(#1)+/dr 0.14cm/^{#3}}
\newcommand{\xdtwocell}[3][0.5]{\ar@{}[#2] \ar@{=>}?(#1)+/u  0.14cm/;?(#1)+/d 0.14cm/^{#3}}

\newcommand{\pushoutcorner}[1][dr]{\save*!/#1+1.2pc/#1:(1,-1)@^{|-}\restore}
\newcommand{\pullbackcorner}[1][dr]{\save*!/#1-1.2pc/#1:(-1,1)@^{|-}\restore}

\newdir{(}{{}*!<0em,-.14em>-\cir<.14em>{l^r}}
\newdir{ (}{{}*!/-5pt/\dir{(}}
\newdir{ >}{{}*!/-5pt/\dir{>}}
\newcommand{\sh}[2]{**{!/#1 -#2/}}

\DeclareMathOperator{\ob}{ob}

\DeclareMathOperator{\im}{im}

\newcommand{\cat}[1]{\mathbf{#1}}
\newcommand{\op}{\mathrm{op}}

\newcommand{\id}{\mathrm{id}}
\newcommand{\thg}{{\mathord{\text{--}}}}

\newcommand{\Ran}{\mathrm{Ran}}

\newcommand{\res}[2]{\left.{#1}\right|_{#2}}

\newcommand{\spn}[1]{{\langle{#1}\rangle}}

\newcommand{\defeq}{\mathrel{\mathop:}=}
\newcommand{\cd}[2][]{\vcenter{\hbox{\xymatrix#1{#2}}}}
\newcommand{\induced}{induced\xspace}

\renewcommand{\phi}{\varphi}
\newcommand{\A}{{\mathcal A}}
\newcommand{\B}{{\mathcal B}}
\newcommand{\C}{{\mathcal C}}
\newcommand{\D}{{\mathcal D}}
\newcommand{\E}{{\mathcal E}}

\newcommand{\G}{{\mathcal G}}

\newcommand{\I}{{\mathcal I}}
\newcommand{\J}{{\mathcal J}}

\newcommand{\M}{{\mathcal M}}
\newcommand{\N}{{\mathcal N}}
\renewcommand{\O}{{\mathcal O}}
\renewcommand{\P}{{\mathcal P}}

\let\sec=\S
\renewcommand{\S}{{\mathcal S}}
\newcommand{\T}{{\mathcal T}}

\newcommand{\V}{{\mathcal V}}

\newcommand{\X}{{\mathcal X}}

\newcommand{\xtor}[1]{\cdl[@1]{{} \ar[r]|-{\object@{|}}^{#1} & {}}}

\makeatletter

\def\hookleftarrowfill@{\arrowfill@\leftarrow\relbar{\relbar\joinrel\rhook}}
\def\twoheadleftarrowfill@{\arrowfill@\twoheadleftarrow\relbar\relbar}
\def\leftbararrowfill@{\arrowdoublefill@{\leftarrow\mkern-5mu}\relbar\mapstochar\relbar\relbar}
\def\Leftbararrowfill@{\arrowdoublefill@{\Leftarrow\mkern-2mu}\Relbar\Mapstochar\Relbar\Relbar}
\def\leftringarrowfill@{\arrowdoublefill@{\leftarrow\mkern-3mu}\relbar{\mkern-3mu\circ\mkern-2mu}\relbar\relbar}
\def\lefttriarrowfill@{\arrowfill@{\mathrel\triangleleft\mkern0.5mu\joinrel\relbar}\relbar\relbar}
\def\Lefttriarrowfill@{\arrowfill@{\mathrel\triangleleft\mkern1mu\joinrel\Relbar}\Relbar\Relbar}

\def\hookrightarrowfill@{\arrowfill@{\lhook\joinrel\relbar}\relbar\rightarrow}
\def\twoheadrightarrowfill@{\arrowfill@\relbar\relbar\twoheadrightarrow}
\def\rightbararrowfill@{\arrowdoublefill@{\relbar\mkern-0.5mu}\relbar\mapstochar\relbar\rightarrow}
\def\Rightbararrowfill@{\arrowdoublefill@{\Relbar\mkern-2mu}\Relbar\Mapstochar\Relbar\Rightarrow}
\def\rightringarrowfill@{\arrowdoublefill@\relbar\relbar{\mkern-2mu\circ\mkern-3mu}\relbar{\mkern-3mu\rightarrow}}
\def\righttriarrowfill@{\arrowfill@\relbar\relbar{\relbar\joinrel\mkern0.5mu\mathrel\triangleright}}
\def\Righttriarrowfill@{\arrowfill@\Relbar\Relbar{\Relbar\joinrel\mkern1mu\mathrel\triangleright}}

\def\leftrightarrowfill@{\arrowfill@\leftarrow\relbar\rightarrow}
\def\mapstofill@{\arrowfill@{\mapstochar\relbar}\relbar\rightarrow}

\renewcommand*\xleftarrow[2][]{\ext@arrow 20{20}0\leftarrowfill@{#1}{#2}}
\providecommand*\xLeftarrow[2][]{\ext@arrow 60{22}0{\Leftarrowfill@}{#1}{#2}}
\providecommand*\xhookleftarrow[2][]{\ext@arrow 10{20}0\hookleftarrowfill@{#1}{#2}}
\providecommand*\xtwoheadleftarrow[2][]{\ext@arrow 60{20}0\twoheadleftarrowfill@{#1}{#2}}
\providecommand*\xleftbararrow[2][]{\ext@arrow 10{22}0\leftbararrowfill@{#1}{#2}}
\providecommand*\xLeftbararrow[2][]{\ext@arrow 50{24}0\Leftbararrowfill@{#1}{#2}}
\providecommand*\xleftringarrow[2][]{\ext@arrow 10{26}0\leftringarrowfill@{#1}{#2}}
\providecommand*\xlefttriarrow[2][]{\ext@arrow 80{24}0\lefttriarrowfill@{#1}{#2}}
\providecommand*\xLefttriarrow[2][]{\ext@arrow 80{24}0\Lefttriarrowfill@{#1}{#2}}

\renewcommand*\xrightarrow[2][]{\ext@arrow 01{20}0\rightarrowfill@{#1}{#2}}
\providecommand*\xRightarrow[2][]{\ext@arrow 04{22}0{\Rightarrowfill@}{#1}{#2}}
\providecommand*\xhookrightarrow[2][]{\ext@arrow 00{20}0\hookrightarrowfill@{#1}{#2}}
\providecommand*\xtwoheadrightarrow[2][]{\ext@arrow 03{20}0\twoheadrightarrowfill@{#1}{#2}}
\providecommand*\xrightbararrow[2][]{\ext@arrow 01{22}0\rightbararrowfill@{#1}{#2}}
\providecommand*\xRightbararrow[2][]{\ext@arrow 04{24}0\Rightbararrowfill@{#1}{#2}}
\providecommand*\xrightringarrow[2][]{\ext@arrow 01{26}0\rightringarrowfill@{#1}{#2}}
\providecommand*\xrighttriarrow[2][]{\ext@arrow 07{24}0\righttriarrowfill@{#1}{#2}}
\providecommand*\xRighttriarrow[2][]{\ext@arrow 07{24}0\Righttriarrowfill@{#1}{#2}}

\providecommand*\xmapsto[2][]{\ext@arrow 01{20}0\mapstofill@{#1}{#2}}
\providecommand*\xleftrightarrow[2][]{\ext@arrow 10{22}0\leftrightarrowfill@{#1}{#2}}
\providecommand*\xLeftrightarrow[2][]{\ext@arrow 10{27}0{\Leftrightarrowfill@}{#1}{#2}}

\makeatother

\numberwithin{equation}{section}

\theoremstyle{plain}
\newtheorem{Thm}{Theorem}
\newtheorem*{Thm*}{Theorem}
\newtheorem{Prop}[Thm]{Proposition}
\newtheorem{Cor}[Thm]{Corollary}
\newtheorem{Lemma}[Thm]{Lemma}

\theoremstyle{definition}
\newtheorem{Defn}[Thm]{Definition}

\newtheorem{Ex}[Thm]{Example}
\newtheorem{Exs}[Thm]{Examples}
\newtheorem{Rk}[Thm]{Remark}
\newtheorem{const}{Construction}

\newcommand{\Cat}{\cat{Cat}}

\newcommand{\Gpd}{\cat{Gpd}}
\newcommand{\Set}{\cat{Set}}
\newcommand{\CAT}{\cat{CAT}}

\newcommand{\Alg}[2][\E]{{#1}^{\mathsf{#2}}}
\newcommand{\Kl}[2][\E]{{#1}_{\mathsf{#2}}}

\newcommand{\atwo}{{\mathbf 2}}

\newcommand{\VCat}[1][\V]{{#1\text-\cat{CAT}}}

\newcommand{\Gph}{\cat{Grph}}
\newcommand{\KN}{K\text-\cat{Ner}}

\newcommand{\mnd}{\cat{Mnd}(\E)}
\newcommand{\eend}{\cat{End}(\E)}
\newcommand{\aend}{\A\text-\cat{End}(\E)}
\newcommand{\amnd}{\A\text-\cat{Mnd}(\E)}
\newcommand{\sig}{\cat{Sig}_\A(\E)}
\newcommand{\preth}{\cat{Preth}_\A(\E)}
\renewcommand{\th}{\cat{Th}_\A(\E)}
\newcommand{\nerv}{\cat{Mnd}_\A(\E)}
\newcommand{\conc}[1][\T]{\cat{Mod}_c(#1)}

\setlist{leftmargin=1.75\parindent,itemsep=0.25\baselineskip}

\begin{document}

\title{Monads and theories}
 \date{\today}
\author{John Bourke}
\address{Department of Mathematics and Statistics, Masaryk University, Kotl\'a\v rsk\'a 2, Brno 61137, Czech Republic}
\email{bourkej@math.muni.cz} 
\author{Richard Garner}
\address{Department of Mathematics, Macquarie University, NSW 2109, Australia}
\email{richard.garner@mq.edu.au}
\thanks{Both authors gratefully acknowledge the support of Australian Research
  Council Discovery Project DP160101519; the second author also
  acknowledges Australian Research Council Future Fellowship
  FT160100393.}
\maketitle

\begin{abstract}
  Given a locally presentable enriched category $\E$ together with a
  small dense full subcategory $\A$ of \emph{arities}, we study the
  relationship between monads on $\E$ and identity-on-objects functors
  out of $\A$, which we call $\A$-pretheories. We show that the
  natural constructions relating these two kinds of structure form an
  adjoint pair. The fixpoints of the adjunction are characterised on
  the one side as the \emph{$\A$-nervous monads}---those for which the
  conclusions of Weber's \emph{nerve theorem} hold---and on the other,
  as the \emph{$\A$-theories} which we introduce here.
  
  The resulting equivalence between $\A$-nervous monads and
  $\A$-theories is best possible in a precise sense, and extends
  almost all previously known monad--theory correspondences. It also
  establishes some completely new correspondences, including one which
  captures the globular theories defining Grothendieck weak
  $\omega$-groupoids.
  
  Besides establishing our general correspondence and illustrating its
  reach, we study good properties of $\A$-nervous monads and
  $\A$-theories that allow us to recognise and construct them with
  ease. We also compare them with the monads with arities and theories
  with arities introduced and studied by Berger, Melli\`es and Weber.
\end{abstract}

\section{Introduction}
\label{sec:introduction}

Category theory provides two approaches to classical universal
algebra. On the one hand, we have finitary monads on $\Set$ and on the
other hand, we have Lawvere theories. Relating the two approaches we
have Linton's result~\cite{Linton1966Some}, which shows that the
category of finitary monads on $\Set$ is equivalent to the category of
Lawvere theories. An essential feature of this equivalence is that it
respects semantics, in the sense that the algebras for a finitary
monad coincide up to equivalence over $\Set$ with the models of the
associated theory, and vice versa.

There have been a host of generalisations of the above story, each
dealing with algebraic structure borne by objects more general than
sets. In many of these~\cite{Power1999Enriched, Nishizawa2009Lawvere,
  Lack2009Gabriel-Ulmer, Lack2011Notions}, one starts on one side with
the monads on a given category that preserve a specified class of
colimits. This class specifies, albeit indirectly, the arities of
operations that may arise in the algebraic structures encoded by such
monads, and from this one may define, on the other side, corresponding
notions of theory and model. These are subtler than in the classical
setting, but once the correct definitions have been found, the
equivalence with the given class of monads, and the compatibility with
semantics, follows much as before.

The most general framework for a monad--theory correspondence to date
involves the notions of \emph{monad with arities} and \emph{theory
  with arities}. In this setting, the permissible arities of
operations are part of the basic data, given as a small, dense, full
subcategory of the base category. The monads with arities were
introduced first, in~\cite{Weber2007Familial}, as a setting for an
abstract nerve theorem. Particular cases of this theorem include the
classical nerve theorem, identifying categories with simplicial sets
satisfying the Segal condition of~\cite{Segal1968Classifying}, and
also Berger's nerve theorem~\cite{Berger2002A-Cellular} for the
globular higher categories of~\cite{Batanin1998Monoidal}. More
saliently, when Weber's nerve theorem is specialised to the settings
appropriate to the monad--theory correspondences listed above, it
becomes exactly the fact that the functor sending the algebras for a
monad to the models of the associated theory is an equivalence. This
observation led \cite{Mellies2010Segal} and~\cite{Berger2012Monads} to
introduce \emph{theories with arities}, and to prove, by using Weber's
nerve theorem, their equivalence with the monads with arities. The
monad--theory correspondence obtained in this way is general enough to
encompass all of the instances from~\cite{Power1999Enriched,
  Nishizawa2009Lawvere, Lack2009Gabriel-Ulmer, Lack2011Notions}.

Our own work in this paper has two motivations: one abstract and one
concrete. Our abstract motivation is a desire to explain the
apparently \emph{ad hoc} design choices involved in the monad--theory
correspondences outlined above. For indeed, while these choices must
be carefully balanced in order to obtain an equivalence, there is no
reason to believe that different careful choices might not yield more
general or more expressive results.

Our concrete motivation comes from the study of the Grothendieck weak
$\omega$-groupoids introduced by
Maltsiniotis~\cite{Maltsiniotis2010Grothendieck}, which, by
definition, are models of a globular theory in the sense of
Berger~\cite{Berger2002A-Cellular}. Globular theories describe
algebraic structure on globular sets with arities drawn from the dense
subcategory of \emph{globular cardinals}; see
Example~\ref{ex:2}\ref{item:2} below. However, globular theories are
not necessarily theories with arities, and in particular, those
capturing higher groupoidal structures are not. As such, they do not
appear to one side of any of the monad--theory correspondences
described above.
 
The first goal of this paper is to describe a new schema for
monad--theory correspondences which addresses the gaps in our
understanding noted above. In this schema, once we have fixed the
process by which a theory is associated to a monad, everything else is
forced. This addresses our first, abstract motivation. The
correspondence obtained in this way is in fact \emph{best possible},
in the sense that any other monad--theory correspondence for the same
kind of algebraic structure must be a restriction of this particular
one. In many cases, this best possible correspondence coincides with
one in the literature, but in others, our correspondence goes beyond
what already exists. In particular, an instance of our schema will
identify the globular theories of~\cite{Berger2002A-Cellular} with a
suitable class of monads on the category of globular sets. This
addresses our second, concrete motivation.

The further goal of this paper is to study the classes of monads and
theories that arise from our correspondence-schema. We do so both at a
general level, where we will see that both the monads and the theories
are closed under essentially all the constructions one could hope for;
and also at a practical level, where we will see how these general
constructions allow us to give expressive and intuitive
\emph{presentations} for the structure captured by a monad or theory.

To give a fuller account of our results, we must first describe how a
typical monad--theory correspondence arises. As
in~\cite{Weber2007Familial}, the basic setting for such a
correspondence can be encapsulated by a pair consisting of a category
$\E$ and a small, full, dense subcategory
$K \colon \A \hookrightarrow \E$. For example, the Lawvere
theory--finitary monad correspondence for finitary algebraic structure
on sets is associated to the choice of $\E = \cat{Set}$ and
$\A = \mathbb{F}$ the full subcategory of finite cardinals.

Given $\E$ and $K \colon \A \hookrightarrow \E$, the goal is to
establish an equivalence between a suitable category of
\emph{$\A$-monads} and a suitable category of \emph{$\A$-theories}.
The $\A$-monads will be a certain class of monads on $\E$; while the
$\A$-theories will be a certain class of identity-on-objects functors
out of $\A$. We are being deliberately vague about the conditions on
each side, as they are among the seemingly \emph{ad hoc} design
choices we spoke of earlier. But regardless of this, the monad--theory
correspondence itself always arises through application of the
following two constructions.

\begin{const}
  \label{const:A}
  For an $\A$-monad $\mathsf{T}$ on $\E$, the associated $\A$-theory
  $\Phi(\mathsf{T})$ is the identity-on-objects functor
  $J_\mathsf{T} \colon \A \rightarrow \A_\mathsf{T}$ arising from the
  (identity-on-objects, fully faithful) factorisation
  \begin{equation}\label{eq:19}
    \A \xrightarrow{J_\mathsf{T}} \A_\mathsf{T}
    \xrightarrow{V_\mathsf{T}}  \E_\mathsf{T}
  \end{equation}
  of the composite
  $F_\mathsf{T} K \colon \A \rightarrow \E \rightarrow \E_\mathsf{T}$.
  Here $F_\mathsf{T}$ is the free functor into the Kleisli category
  $\E_\mathsf{T}$, so $\A_\mathsf{T}$ is equally the full
  subcategory of $\E_\mathsf{T}$ with objects those of $\A$.
\end{const}  
 
\begin{const}
  \label{const:B}
  For an $\A$-theory $J \colon \A \rightarrow \T$, the associated
  $\A$-monad $\Psi (\T)$ is obtained from the category of
  \emph{concrete $\T$-models}, which is by definition the
  pullback \begin{equation}\label{eq:20} \cd{ {\conc} \pullbackcorner
      \ar[rr]^-{} \ar[d]_{U^{\T}} &&
      {[\T^\mathrm{op}, \cat{Set}]} \ar[d]^{[J^\mathrm{op}, 1]} \\
      {\E} \ar[rr]^-{N_K=\E(K-,1)} && {[\A^\mathrm{op},
        \cat{Set}]}\rlap{ .} }
  \end{equation}
  Since $U^{\T}$ is a pullback of the strictly monadic
  $[J^\mathrm{op}, 1]$, it will be strictly monadic so long as it has
  a left adjoint. The assumption that $\E$ is locally
    presentable ensures that this is the case, and so we can take
  $\Psi(\T)$ to be the monad whose algebras are the concrete
  $\T$-models.
\end{const}

There remains the problem of choosing the appropriate conditions on a
monad or theory for it to be an $\A$-monad or $\A$-theory. Of course,
these must be carefully balanced so as to obtain an equivalence, but
this still seems to leave too many degrees of freedom; one might hope
that everything could be determined from $\E$ and $\A$ alone. The main
result of this paper shows that this is so: there are notions of
$\A$-monad and $\A$-theory which require no further choices to be
made, and which rather than being plucked from the air, may be derived
in a principled manner.

The key observation is that Constructions~\ref{const:A}
and~\ref{const:B} make sense when given as input \emph{any} monad on
$\E$, or \emph{any} ``$\A$-pretheory''---by which we mean simply an
identity-on-objects functor out of $\A$. When viewed in this greater
generality, these constructions yield an adjunction
\begin{equation}\label{eq:7}
    \cd[@C+1em]{
    {\mnd} \ar@<-4.5pt>[r]_-{\Phi} \ar@{}[r]|-{\bot} &
    {\preth} \ar@<-4.5pt>[l]_-{\Psi}
  }
\end{equation}
between the category of monads on $\E$ and the category of
$\A$-pretheories. Like any adjunction, this restricts to an
equivalence between the objects at which the counit is invertible, and
the objects at which the unit is invertible. Thus, if we \emph{define}
the $\A$-monads and $\A$-theories to be the objects so arising, then
we obtain a monad--theory equivalence. By construction, it will be the
\emph{largest} possible equivalence whose two directions are given by
Constructions~\ref{const:A} and~\ref{const:B}.

Having defined the $\A$-monads and $\A$-theories abstractly, it
behooves us to give tractable concrete characterisations. In fact, we
give a number of these, allowing us to relate our correspondence to
existing ones in the literature. We also investigate further aspects
of the general theory, and provide a wide range of examples
illustrating the practical utility of our results.

Before getting started, we conclude this introduction with a more
detailed outline of the paper's contents. In
Section~\ref{sec:monads-pretheories}, we begin by introducing our
basic setting and notions. We then construct, in
Theorem~\ref{thm:adjunction}, the adjunction~\eqref{eq:7} between
monads and pretheories. In Section~\ref{sec:preth-as-pres}, with this
abstract result in place, we introduce a host of running examples of
our basic setting. To convince the reader of the expressive power of
our notions, we construct, via colimit presentations, specific
pretheories for a variety of mathematical structures.

In Section~\ref{sec:monad-theory-corr} we obtain our main result by
characterising the fixpoints of the monad--theory adjunction: the
$\A$-monads and $\A$-theories described above. The $\A$-monads are
characterised as what we term the \emph{$\A$-nervous monads}, since
they are precisely those monads for which Weber's nerve theorem holds.
The $\A$-theories turn out to be precisely those $\A$-pretheories for
which each representable is a model; in the motivating case where
$\E = \cat{Set}$ and $\A = \mathbb{F}$, they are exactly the Lawvere
theories. With these characterisations in place, we obtain our main
Theorem~\ref{thm:1}, which describes the ``best possible'' equivalence
between $\A$-theories and $\A$-nervous monads.

Section~\ref{sec:semantics} develops some of the general results
associated to our correspondence-schema. We begin by showing that our
monad--theory correspondence commutes, to within isomorphism, with the
taking of semantics on each side. We also prove that the functors
taking semantics are valued in monadic right adjoint functors between
locally presentable categories. The final important result of this
section states that colimits of $\A$-nervous monads and $\A$-theories
are algebraic, meaning that the semantics functors send them to
limits.

Section~\ref{sec:identifying-theories} is devoted to exploring what
the $\A$-nervous monads and $\A$-theories amount to in our running
examples. In order to understand the $\A$-nervous monads, we prove the
important result that they are equally the colimits, amongst all
monads, of free monads on $\A$-signatures. We also introduce the
notion of a \emph{saturated class of arities} as a setting in
which,~like in \cite{Power1999Enriched, Nishizawa2009Lawvere,
  Lack2009Gabriel-Ulmer, Lack2011Notions}, the $\A$-nervous monads can
be characterised in terms of a colimit-preservation property. With
these results in place, we are able to exhibit many of these existing
monad--theory correspondences as instances of our general framework.

In Section~\ref{sec:arities}, we examine the relationship between the
monads and theories of our correspondence, and the \emph{monads with
  arities} and \emph{theories with arities}
of~\cite{Weber2007Familial, Mellies2010Segal, Berger2012Monads}. In
particular, we see that every monad with arities $\A$ is an
$\A$-nervous monad but that the converse implication need not be true:
so $\A$-nervous monads are \emph{strictly} more general. Of course,
the same is also true on the theory side. We also exhibit a further
important point of difference: colimits of monads with arities, unlike
those of nervous monads, need not be algebraic. This means that there
is no good notion of \emph{presentation} for monads or theories with
arities.

Finally, in Section~\ref{sec:deferred-proofs}, we give a number of
proofs deferred from Section~\ref{sec:identifying-theories}.

\section{Monads and pretheories}
\label{sec:monads-pretheories}

\subsection{The setting}
\label{sec:setting}
In this section we construct the monad--pretheory
adjunction
\begin{equation}
  \label{eq:adj}
      \cd[@C+1em]{
    {\mnd} \ar@<-4.5pt>[r]_-{\Phi} \ar@{}[r]|-{\bot} &
    {\preth\rlap{ .}} \ar@<-4.5pt>[l]_-{\Psi}
  }
\end{equation}
The setting for this, and the rest of the paper, involves two basic
pieces of data:
\begin{enumerate}[(i)]
\item A locally presentable $\V$-category $\E$ with respect to which
  we will describe the monad--pretheory adjunction; and
\item A notion of \emph{arities} given by a small, full, dense
  sub-$\V$-category $K \colon \A \hookrightarrow \E$.
\end{enumerate}

We will discuss examples in Section~\ref{sec:setting} below, but for
now let us clarify some of the terms appearing above. While in the
introduction, we focused on the unenriched context, we now work in the
context of category theory enriched over a symmetric monoidal closed
category $\V$ which is \emph{locally presentable} as
in~\cite{Gabriel1971Lokal}. In this context, a \emph{locally
  presentable $\V$-category}~\cite{Kelly1982Structures} is one which
is cocomplete as a $\V$-category, and whose underlying ordinary
category is locally presentable.

We recall also some notions pertaining to density. Given a
$\V$-functor $K \colon \A \rightarrow \E$ with small domain, the
\emph{nerve functor}
${N_K \colon \E \rightarrow [\A^\mathrm{op}, \V]}$ is defined by
$N_K(X) = \E(K\thg, X)$. We call a presheaf in the essential image of
$N_K$ a \emph{$K$-nerve}, and we write $\KN(\V)$ for the full
sub-$\V$-category of $[\A^\mathrm{op}, \V]$ determined by these.

We say that $K$ is \emph{dense} if $N_K$ is fully faithful; whereupon
$N_K$ induces an equivalence of categories $\E \simeq \KN(\V)$.
Finally, we call a small sub-$\V$-category $\A$ of a $\V$-category
$\E$ \emph{dense} if its inclusion functor
$K \colon \A \hookrightarrow \E$ is so.

\subsection{Monads}
\label{sec:monads}

We write $\mnd$ for the (ordinary) category whose objects are
$\V$-monads on $\E$, and whose maps $\mathsf S \to \mathsf T$ are
$\V$-transformations $\alpha \colon S \Rightarrow T$ compatible with
unit and multiplication. For each $\mathsf T \in \mnd$ we have the
$\V$-category of algebras $U^\mathsf T \colon \Alg{T} \to \E$ over
$\E$, but also the \emph{Kleisli $\V$-category}
$F_\mathsf T \colon \E \to \Kl{T}$ under~$\E$, arising from an
(identity-on-objects, fully faithful) factorisation
\begin{equation}
  \cd[@-0.8em]{
    & {\E} \ar[dl]_-{F_\mathsf{T}} \ar[dr]^-{F^\mathsf{T}} \\
    {\E_{\mathsf{T}}} \ar[rr]^-{W_\mathsf{T}} & &
    {\E^{\mathsf{T}}}}
\end{equation}
of the free $\V$-functor $F^{\mathsf T} \colon \E \to \Alg{T}$;
concretely, we may take $\Kl{T}$ to have objects those of $\E$,
hom-objects $\Kl{T}(A,B) = \E(A,TB)$, and composition and identities
derived using the monad structure of $\mathsf T$. Each monad morphism
$\alpha \colon \mathsf S \to \mathsf T$ induces, functorially in
$\alpha$, $\V$-functors $\alpha^\ast$ and $\alpha_!$ fitting into
diagrams
\begin{equation}
  \cd[@C+0.5em@-1em]{ \Alg{T} \ar[rr]^-{\alpha^\ast}
    \ar[dr]_{U^\mathsf T} & & \Alg{S}
    \ar[dl]^{U^\mathsf S} & & & \E \ar[dl]_{F_\mathsf S}
    \ar[dr]^{F_\mathsf T} \\ & \E & & & \Kl{S}
    \ar[rr]^-{\alpha_!} & & \Kl{T}\rlap{ ;} }
\end{equation}
here $\alpha^\ast$ sends an algebra $a \colon TA \to A$ to
$a \circ \alpha_A \colon SA \to A$ and is the identity on homs, while
$\alpha_!$ is the identity on objects and has action on homs given by
the postcomposition maps
$\alpha_B \circ (\thg) \colon \Kl{S}(A,B) \rightarrow \Kl{T}(A,B)$. In
fact, every $\V$-functor $\Alg{T} \to \Alg{S}$ over $\E$ or
$\V$-functor $\Kl{S} \to \Kl{T}$ under $\E$ is of the form
$\alpha^\ast$ or $\alpha_!$ for a unique map of monads $\alpha$---see,
for example,~\cite{Meyer1975Induced}---and in this way, we obtain
\emph{fully faithful} functors
\begin{equation}\label{eq:10}
  \mnd^\mathrm{op} \xrightarrow{\mathrm{Alg}}
  \V\text-\cat{CAT}/\E \qquad \text{and} \qquad 
  \mnd \xrightarrow{\mathrm{Kl}} \E / \V\text-\cat{CAT}\rlap{ .}
\end{equation}

\subsection{Pretheories}
\label{sec:pretheories}

An \emph{$\A$-pretheory} is an identity-on-objects $\V$-functor
$J \colon \A \rightarrow \T$ with domain $\A$. We write $\preth$ for
the ordinary category whose objects are $\A$-pretheories and whose
morphisms are $\V$-functors commuting with the maps from $\A$. While
the $\A$-pretheory is only fully specified by both pieces of data $\T$
and $J$, we will often, by abuse of notation, leave $J$ implicit and
refer to such a pretheory simply as $\T$.

Just as any $\V$-monad has a $\V$-category of algebras, so any
$\A$-pretheory has a $\V$-category of models.
Generalising~\eqref{eq:20}, we define the $\V$-category of
\emph{concrete $\T$-models} $\conc$ by a pullback of $\V$-categories
as below left; so a concrete $\T$-model is an object $X \in \E$
together with a chosen extension of
$\E(K\thg, X) \colon \A^\mathrm{op} \rightarrow \V$ along
$J^\mathrm{op} \colon \A^\mathrm{op} \rightarrow \T^\mathrm{op}$. The
reason for the qualifier ``concrete'' will be made clear in
Section~\ref{sec:generalised-models} below, where we will identify a
more general notion of model.
\begin{equation}\label{eq:15}
  \cd[@C-0.5em]{
    {\conc} \pullbackcorner \ar[r]^-{P_\T} \ar[d]_{U_\T} &
    {[\T^\mathrm{op}, \V]} \ar[d]^{[J^\mathrm{op}, 1]} &
    {\conc[\S]} \ar[r]^-{P_{\S}} \ar[d]_{U_{\S}} &
    {[\S^{\mathrm{op}}, \V]} \ar[r]^-{[H^\mathrm{op}, 1]} &
    {[\T^\mathrm{op}, \V]} \ar[d]^{[J^\mathrm{op}, 1]} \\
    {\E} \ar[r]^-{N_K} &
    {[\A^\mathrm{op}, \V]} &
    {\E} \ar[rr]^-{N_K} &&
    {[\A^\mathrm{op}, \V]}
  }
\end{equation}

\begin{Rk}
  \label{rk:1}
  Avery considers a notion very similar to our $\A$-pretheories under
  the name
  \emph{prototheories}~\cite[Definition~4.1.1]{Avery2017Structure}.
  The differences are that Avery's prototheories $\A \rightarrow \T$
  are not enriched, and the hom-sets of $\T$ need not be small. He
  also defines a category of (concrete) models for a prototheory,
  relative to a given functor $\E \rightarrow [\A^\mathrm{op}, \C]$
  called an \emph{aritation}. When this functor is the nerve
  $N_K \colon \E \rightarrow [\A^\mathrm{op}, \cat{Set}]$, his
  category of models agrees with our $\cat{Mod}_c(\T)$.
\end{Rk}

Any $\A$-pretheory map $H \colon \T \rightarrow \S$ gives a functor
${H^\ast \colon \conc[\S] \rightarrow \conc}$ over $\E$ by applying
the universal property of the pullback left above to the commuting
square on the right. In this way, we obtain a semantics functor:
\begin{equation}\label{eq:21}
  \preth^\mathrm{op} \xrightarrow{\mathrm{Mod}_c}
  \V\text-\cat{CAT} / \E\rlap{ .}
\end{equation}
However, unlike~\eqref{eq:10}, this is \emph{not} always fully
faithful. Indeed, in Example~\ref{ex:6} below, we will see that
non-isomorphic pretheories can have isomorphic categories of concrete
models over~$\E$.

\subsection{Monads to pretheories}
\label{sec:from-monads-preth}

We now define the functor $\Phi \colon \mnd \rightarrow \preth$
in~\eqref{eq:adj}. As in Construction~\ref{const:A} of the
introduction, this will take the $\V$-monad $\mathsf{T}$ to the
$\A$-pretheory $J_\mathsf{T} \colon \A \rightarrow \A_\mathsf{T}$
arising as the first part of an (identity-on-objects, fully faithful)
factorisation of $F_\mathsf{T} K \colon \A \rightarrow \Kl{T}$, as to
the left in:
\begin{equation}\label{eq:factorisations}
  \cd{
    \A \ar[d]_{K} \ar[r]^{J_{\mathsf T}} & \A_{\mathsf T}
    \ar[d]^{V_{\mathsf T}} & & \A \ar[d]_{K} \ar[r]^{J_{\mathsf T}} & \A_{\mathsf T} \ar[d]^{K_{\mathsf T}} \\
    \E \ar[r]^-{F_{\mathsf T}} & \Kl{T} & &
    \E \ar[r]^-{F^{\mathsf T}} & \Alg{T}\rlap{ .} }
\end{equation}
Since the comparison $W_{\mathsf T} \colon \Kl{T} \to \Alg{T}$ is
fully faithful, we can also view $J_{\mathsf T}$ as arising from an
(identity-on-objects, fully faithful) factorisation as above right;
the relationship between the two is that
$K_{\mathsf T} = W_{\mathsf T} \circ V_{\mathsf T}$. Both perspectives
will be used in what follows, with the functor
$K_{\mathsf T} \colon \A_{\mathsf T} \to \Alg{T}$ of particular
importance.

To define $\Phi$ on morphisms, we make use of the \emph{orthogonality}
of identity-on-objects $\V$-functors to fully faithful ones; this
asserts that any commuting square of $\V$-functors as below, with $F$
identity-on-objects and $G$ fully faithful, admits a unique diagonal
filler $J$ making both triangles commute.
\begin{equation*}
  \cd{
    {\A} \ar[r]^-{H} \ar[d]_{F} &
    {\C} \ar[d]^{G} \\
    {\B} \ar[r]^-{K} \ar@{-->}[ur]^-{J}  &
    {\D}
  }
\end{equation*}
Explicitly, $J$ is given on objects by $Ja = Ha$, and on homs by
\begin{equation*}
  \B(a,b) \xrightarrow{K_{a,b}} \D(Ka, Kb) = \D(GHa, GHb)
  \xrightarrow{(G_{Ha, Hb}{)}^{-1}} \C(Ha, Hb)\rlap{ .}
\end{equation*}

In particular, given a map $\alpha \colon \mathsf S \to \mathsf T$ of
$\mnd$, this orthogonality guarantees the existence of a diagonal
filler in the diagram below, whose upper triangle we take to be the
map
$\Phi(\alpha) \colon \Phi(\mathsf{S}) \rightarrow \Phi(\mathsf{T})$ in
$\preth$:
\begin{equation*}
  \cd{
    \A \ar[d]_{J_{\mathsf{S}}} \ar[rr]^{J_{\mathsf T}} & & \A_{\mathsf T}
    \ar[d]^{V_{\mathsf T}} \\
    \A_{\mathsf S} \ar[r]_{V_{\mathsf S}} \ar@{-->}[urr]^{} & \Kl{S} \ar[r]_{\alpha_!} &
    \Kl{T}\rlap{ .}
  }
\end{equation*}

\subsection{Pretheories to monads}
\label{sec:monad-preth-adjunct-1}

Thus far we have not exploited the local presentability of $\E$. It
\emph{will} be used in the next step, that of constructing the left
adjoint to $\Phi \colon \mnd \rightarrow \preth$. We first state a
general result which, independent of local presentability, gives a
sufficient condition for an individual pretheory to have a reflection
along $\Phi$. Here, by a \emph{reflection} of an object $c \in \C$
along a functor $U \colon \B \rightarrow \C$, we mean a representation
for the functor $\C(c, U\thg) \colon \B \rightarrow \cat{Set}$.

\begin{Thm}
  \label{thm:adjunction0}
  A pretheory $J \colon \A \to \T$ admits a reflection along $\Phi$
  whenever the forgetful functor $U_{\T} \colon \conc \to \E$ from the
  category of concrete models has a left adjoint $F_{\T}$. In this
  case, the reflection $\Psi \T$ is characterised by an isomorphism
  $\E^{\Psi \T} \cong \conc$ over $\E$, or equally, by a pullback
  square
  \begin{equation}
    \cd{
      \Alg{\mathrm{\Psi} \T} \ar[r]^-{} \ar[d]_{U^{\Psi \T}}
      &
      [\T^\op, \V] \ar[d]^{[J^\mathrm{op},1]}\\
      \E \ar[r]^-{N_K} & [\A^\mathrm{op}, \V]\rlap{ .}
    }\label{eq:defL}
  \end{equation}
\end{Thm}

To prove this result, we will need a preparatory lemma, relating to
the notion of \emph{discrete isofibration}: this is a $\V$-functor
$U \colon \D \rightarrow \C$ such that, for each $f \colon c \cong Ud$
in $\C$, there is a unique $f^{\prime} \colon c^{\prime} \cong d$ in
$\D$ with $U(f^{\prime})=f$.

\begin{Ex}
  \label{ex:7}
  For any $\V$-monad $\mathsf{T}$ on $\C$, the forgetful $\V$-functor
  $U^\mathsf{T} \colon \C^\mathsf{T} \rightarrow \C$ is a discrete
  isofibration. Indeed, if $x \colon Ta \rightarrow a$ is a
  $\mathsf{T}$-algebra and $f \colon b \cong a$ in $\C$, then
  $y = f^{-1} \circ x \circ Tf \colon Tb \rightarrow b$ is the unique
  algebra structure on $b$ for which
  $f \colon (b,y) \rightarrow (a,x)$ belongs to $\C^\mathsf{T}$. In
  particular, for any identity-on-objects $\V$-functor
  $F \colon \A \rightarrow \B$ between small $\V$-categories, the
  functor $ [F,1] \colon [\B, \V] \rightarrow [\A, \V]$ has a left
  adjoint and strictly creates colimits, whence is strictly monadic.
  It is therefore a discrete isofibration by the above argument.
\end{Ex}

\begin{Lemma}\label{lem:cod2}
  Let $U \colon \A \rightarrow \B$ be a discrete isofibration and
  $\alpha \colon F \Rightarrow G \colon \X \rightarrow \B$ an
  invertible $\V$-transformation. The displayed projections give
  isomorphisms between liftings of $F$ through $U$, liftings of
  $\alpha$ through $U$, and liftings of $G$ through $U$:
  \begin{equation*}
    \cd[@+1em]{
      & \A \ar[d]^-{U} \\
      \X \ar[r]^-{F} \ar@{-->}[ur]^-{\bar F} & \B
    } \quad \xleftarrow{\mathrm{dom}} \quad
    \cd[@+1.2em]{
      & \A \ar[d]^-{U} \\
      \X \ar@/_0.7em/[r]_-{G} \xdtwocell{r}{\alpha} \xdrtwocell{ur}{
        \bar \alpha} \ar@/^0.7em/[r]^(0.65){F}
      \ar@/^0.7em/@{-->}[ur]^-{\bar F}
      \ar@/_0.7em/@{-->}[ur]_(0.6){\bar G} & \B
    } \quad \xrightarrow{\mathrm{cod}} \quad
    \cd[@+1em]{
      & \A \ar[d]^-{U} \\
      \X \ar[r]^-{G} \ar@{-->}[ur]^-{\bar G} & \B\rlap{ .}
    }
  \end{equation*}
\end{Lemma}

\begin{proof}
  Given $\bar G \colon \X \rightarrow \A$ as to the right, there is
  for each $x \in \X$ a unique lifting of the isomorphism
  $\alpha_x \colon Fx \cong U\smash{\bar G}x$ to one
  $\smash{\bar\alpha}_x \colon \smash{\bar F}x \cong \smash{\bar G}x$.
  There is now a unique way of extending $x \mapsto \smash{\bar F}x$
  to a $\V$-functor $\smash{\bar F} \colon \X \rightarrow \A$ so that
  $\bar \alpha \colon \bar F \cong \bar G$; namely, by taking the
  action on homs to be
  $\bar F_{x,y} = \A(\smash{\bar \alpha}_x, {\smash{\bar
      \alpha}_y}^{-1}) \circ \smash{\bar G}_{x,y} \colon X(x,y)
  \rightarrow \A(\smash{\bar F}x,\smash{\bar F}y)$. In this way, we
  have found a unique lifting of $\alpha$ through $U$ whose codomain
  is the given lifting of $G$ through $U$. So the right-hand
  projection is invertible; the argument for the left-hand one is the
  same on replacing $\alpha$ by $\alpha^{-1}$.
\end{proof}

We can now give:

\begin{proof}[Proof of Theorem~\ref{thm:adjunction0}]
  $U_\T$ has a left adjoint by assumption, and---as a pullback of the
  strictly monadic
  $[J^\mathrm{op},1]\colon [\T^\mathrm{op},\V] \to
  [\A^\mathrm{op},\V]$---strictly creates coequalisers for
  $U_{\T}$-absolute pairs. It is therefore strictly monadic. Taking
  $\Psi \T = U_\T F_\T$ to be the induced monad, we thus have an
  isomorphism $\E^{\Psi \T} \cong \conc$ over $\E$.

  It remains to exhibit isomorphisms
  $\mnd(\Psi \T, \mathsf{S}) \cong \preth(\T, \Phi \mathsf{S})$
  natural in $\mathsf{S}$. We do so by chaining together the following
  sequence of natural bijections. Firstly, by full fidelity
  in~\eqref{eq:10}, monad maps
  $\alpha_{0} \colon \Psi \T \to \mathsf{S}$ correspond naturally to
  functors $\alpha_{1}\colon \Alg{S} \to \Alg{\mathrm{\Psi}\T}$
  rendering commutative the left triangle in
  \begin{equation}\label{eq:13}
    \cd[@!C@C-3em]{
      \Alg{S} \ar[dr]_{U^{\mathsf S}} \ar[rr]^{\alpha_{1}} &&
      \Alg{\mathrm{\Psi}\T} \ar[dl]^{U^{\Psi \T}} && \Alg{S}
      \ar[rr]^-{\alpha_{2}} \ar[d]_{U^{\mathsf S}} && [\T^\mathrm{op},\V] \ar[d]^{[J^\mathrm{op},1]} \\
      & \E &&& \E \ar[rr]^-{N_{K}} && [\A^\mathrm{op},\V]\rlap{ .}
    }
  \end{equation}
  Since $\Alg{\mathrm{\Psi} \T}$ is defined by the
  pullback~\eqref{eq:defL}, such functors $\alpha_1$ correspond
  naturally to functors $\alpha_{2}$ rendering commutative the square
  above right. Next, we observe that there is a natural isomorphism in
  the triangle below left
  \begin{equation}\label{eq:14}
    \cd[@!C@C-2em]{
      \Alg{S} \ar[rr]^-{U^{\mathsf S}} \ar[dr]_-{N_{F^\mathsf{S}K}} & \twocong{d} &
      {\E} \ar[dl]^-{N_{K}} &
      \Alg{S} \ar[rr]^-{\alpha_{3}} \ar[dr]_-{N_{F^\mathsf{S}K}} && [\T^{\op},\V] \ar[dl]^-{[J^{\op},1]} \\
      & [\A^\mathrm{op},\V] && 
      & [\A^{\op},\V]}
  \end{equation}
  with components the adjointness isomorphisms
  $\E(Ka,U^{\mathsf S}b) \cong \Alg{S}(F^{\mathsf S}Ka,b)$. Since
  $J^\mathrm{op}$ is identity-on-objects, $[J^\mathrm{op}, 1]$ is a
  discrete isofibration by Example~\ref{ex:7}, whence by
  Lemma~\ref{lem:cod2} there is a natural bijection between functors
  $\alpha_2$ as in~\eqref{eq:13} and ones $\alpha_3$ as
  in~\eqref{eq:14}. We should now like to transpose this last triangle
  through the following natural isomorphisms (taking $\X = \A, \T$):
  \begin{equation}\label{eq:31}
    \VCat(\Alg{S},[\X^\mathrm{op},\V]) \cong \VCat(\X^\mathrm{op},[\Alg{S},\V])\rlap{ .}
  \end{equation}
  However, since $\E^\mathsf{S}$ is large, the
  functor category $[\E^\mathsf{S}, \V]$ will not always exist as a
  $\V$-category, and so~\eqref{eq:31} is ill-defined. To resolve
  this, note that $N_{F^\mathsf{S} K}$ is, by its definition,
  \emph{pointwise representable}; whence so too is $\alpha_3$, since $J$
  is identity-on-objects. We may thus transpose the right triangle
  of~\eqref{eq:14} through the legitimate isomorphisms
  \begin{equation}\label{eq:33}
    \VCat(\Alg{S},[\X^\mathrm{op},\V])_{\mathrm{pwr}} \cong
    \VCat(\X^\mathrm{op},[\Alg{S},\V]_\mathrm{rep})
  \end{equation}
  where on the left we have the category of pointwise representable
  $\V$-functors, and on the right, the \emph{legitimate} $\V$-category
  of representable $\V$-functors $\E^\mathsf{S} \rightarrow \V$.
  In
  this way, we establish a natural bijection between functors
  $\alpha_3$ and functors $\alpha_4$ rendering commutative the left
  square in:
  \begin{equation*}
    \cd[@!C]{
      \A^\mathrm{op} \ar[d]_{(F^{\mathsf S}K)^\mathrm{op}}
      \ar[r]^{J^\mathrm{op}} &
      \T^\mathrm{op} \ar[d]^{\alpha_{4}} \ar@{.>}[dl]|{(\alpha_{5})^\mathrm{op}} &&
      \A \ar[d]_{J_{\mathsf S}} \ar[r]^{J} &
      \T \ar[d]^{{\alpha_{5}}} \ar@{.>}[dl]|{\alpha_{6}} \\
      (\Alg{S})^\mathrm{op} \ar[r]_{Y} &
      \sh{r}{0.5em}{[\Alg{S},\V]\rlap{${}_\mathrm{rep}$}} &&
      \A_\mathsf{S} \ar[r]_{K_{\mathsf S}} & \Alg{S}\rlap{ .}
    }
  \end{equation*}
  Now orthogonality of the identity-on-objects $J^\mathrm{op}$ and the
  fully faithful $Y$ draws the correspondence between functors
  $\alpha_{4}$ and functors $\alpha_{5}$ satisfying
  $\alpha_5 \circ J = F^\mathsf{S}K$ as left above. Finally, since
  $\A_{\mathsf S}$ fits in to an (identity-on-objects, fully faithful)
  factorisation of $F^{\mathsf S}K$, orthogonality also gives the
  correspondence, as right above, between functors $\alpha_{5}$ and
  functors $\alpha_{6}$ satisfying $\alpha_6 \circ J = J_{\mathsf S}$,
  as required.
\end{proof}

We now show that the assumed local presentability of $\E$ ensures that
every pretheory has a reflection along $\Phi \colon \mnd \to \preth$,
which consequently has a left adjoint. The key result about locally
presentable categories enabling this is the following lemma.

\begin{Lemma}
  \label{lem:1}
  Consider a pullback square of $\V$-categories
  \begin{equation}
    \label{eq:6}
    \cd[@-0em]{
      \A \pullbackcorner \ar[r]^{F} \ar[d]_{U} & \B
      \ar[d]^{V}\\
      \C \ar[r]^-{G} & \D
    }
  \end{equation}
  in which $G$ and $V$ are right adjoints between locally presentable
  $\V$-categories and $V$ is strictly monadic. Then $U$ and $F$ are
  right adjoints between locally presentable $\V$-categories and $U$
  is strictly monadic.
\end{Lemma}
\begin{proof}
  Since $V$ is strictly monadic, it is a discrete isofibration, and so
  its pullback against $G$ is,
  by~\cite[Corollary~1]{Joyal1993Pullbacks}, also a bipullback.
  By~\cite[Theorem~6.11]{Bird1984Limits} the $2$-category of locally
  presentable $\V$-categories and right adjoint functors is closed
  under bilimits in $\VCat$, so that both $U$ and $F$ are right
  adjoints between locally presentable categories. Finally, since $U$
  is a pullback of the strictly monadic $V$, it strictly creates
  coequalisers for $U$-absolute pairs. Since it is already known to be
  a right adjoint, it is therefore also strictly monadic.
\end{proof}

With this in place, we can now prove:

\begin{Thm}
 \label{thm:adjunction}
 Let $\E$ be locally presentable. Then $\Phi \colon \mnd \to \preth$
 has a left adjoint $\Psi \colon \preth \rightarrow \mnd$, whose value
 at the pretheory $J \colon \A \to \T$ is characterised by an
 isomorphism $\E^{\Psi(\T)} \cong \conc$ over $\E$, or equally, by a
 pullback square
\begin{equation}
  \cd{
    \Alg{\mathrm{\Psi}(\T)} \ar[r]^-{} \ar[d]_{U^{\Psi(\T)}}
    \pullbackcorner &
    [\T^\op, \V] \ar[d]^{[J^\mathrm{op},1]}\\
    \E \ar[r]^-{N_K} & [\A^\mathrm{op}, \V]\rlap{ .}
  }\label{eq:defL}
\end{equation}
\end{Thm}

\begin{proof}
  Let $J\colon\A \to \T$ be a pretheory. The pullback square
  \eqref{eq:15} defining $\conc$ is a pullback of a right adjoint
  functor between locally presentable categories along a strictly
  monadic one: so it follows from Lemma~\ref{lem:1} that
  $U_{\T}\colon\conc \to \E$ is a right adjoint, whence the result
  follows from Theorem~\ref{thm:adjunction0}.
\end{proof}

\begin{Rk}
  \label{rk:2}
  In Avery's study of prototheories, he establishes a
  \emph{structure--semantics}
  adjunction~\cite[Theorem~4.4.8]{Avery2017Structure} of the form
  $\cat{Proto}_\A(\E)^\mathrm{op} \leftrightarrows \cat{CAT} / \E$,
  where here $\cat{CAT}$ is the category of \emph{large} categories.
  By restricting to the locally small prototheories to the left and to
  the strictly monadic functors to the right of this adjunction, one
  can recover, via~\eqref{eq:10}, the unenriched case of our
  adjunction~\eqref{eq:adj}.
  \end{Rk}

\section{Pretheories as presentations}
\label{sec:preth-as-pres}
In the next section, we will describe how the monad--pretheory
adjunction~\eqref{eq:adj} restricts to an equivalence between suitable
subcategories of \emph{$\A$-theories} and of \emph{$\A$-nervous
  monads}. However, the results we have so far are already practically
useful. The notion of $\A$-pretheory provides a tool for presenting
certain kinds of algebraic structure, by exhibiting them as categories
of concrete $\T$-models for a suitable pretheory in a manner
reminiscent of the theory of \emph{sketches}~\cite{Barr1985Toposes}.
Equivalently, via the functor $\Psi$, we can see $\A$-pretheories as a
way of presenting certain monads on $\E$.

\subsection{Examples of the basic setting}
\label{sec:basic-setting}
Before giving examples of algebraic structures presented by
pretheories, we first describe a range of examples of the basic
setting of Section~\ref{sec:setting} above.

\begin{Exs}\label{ex:2}
  We begin by considering the unenriched case where $\V = \cat{Set}$.
  \begin{enumerate}[(i)]
  \item Taking $\E = \cat{Set}$ and $\A = \mathbb{F}$ the full
    subcategory of finite cardinals captures the classical case of
    \emph{finitary} algebraic structure borne by \emph{sets}; so
    examples like groups, rings, lattices, Lie algebras, and so
    on.\label{item:3}
    
  \item Taking $\E$ a locally finitely presentable category and
    $\A = \E_f$ a skeleton of the full subcategory of finitely
    presentable objects, we capture \emph{finitary} algebraic
    structure borne by \emph{$\E$-objects}. Examples when
    $\E = \cat{Cat}$ include \emph{finite product}, \emph{finite
      colimit}, and \emph{monoidal closed} structure; for
    $\E = \cat{CRng}$, we have \emph{commutative $k$-algebra},
    \emph{differential ring} and \emph{reduced ring}
    structure.\label{item:11}

  \item We can replace ``finitary'' above by ``$\lambda$-ary'' for any
    regular cardinal $\lambda$. For example, when
    $\lambda = \aleph_1$, this allows for the structure of poset with
    joins of $\omega$-chains~\cite{Markowsky1976Chain-complete} when
    $\E = \cat{Set}$, and for \emph{countable product} structure when
    $\E = \cat{Cat}$. When $\E = [\O(X)^\mathrm{op}, \cat{Set}]$ for
    some space $X$, and $\lambda$ is suitably chosen, it also permits
    \emph{sheaf} or \emph{sheaf of rings} structure.\label{item:12}

  \item Let $\mathbb G_{1}$ be the category freely generated by the
    graph $0 \rightrightarrows 1$, so that
    $\E = [\mathbb G_{1}^\mathrm{op},\Set]$ is the category of
    directed multigraphs, and let $\A = \Delta_0$ be the full
    subcategory of $[\mathbb G_{1}^\mathrm{op},\Set]$ on graphs of the
    form
    \begin{equation*} \cd{ [n]\defeq 0 \ar[r]^{} & 1 \ar[r] & \cdots
        \ar[r] & n } \qquad \qquad \text{for $n>0$.}
    \end{equation*}
    $\Delta_0$ is dense in $[\mathbb G_{1}^\mathrm{op},\Set]$ because
    it contains the representables $[0]$ and $[1]$. This example
    captures structure borne by graphs in which the operations build
    vertices and arrows from \emph{paths} of arrows: for example, the
    structures of \emph{categories}, \emph{involutive categories}, and
    \emph{groupoids}.
    \label{item:1}

  \item The \emph{globe category} $\mathbb G$ is freely generated by the
    graph
    \begin{equation*}
      \cd{
        0 \ar@<-3pt>[r]_{\tau} \ar@<3pt>[r]^{\sigma} & 1
        \ar@<-3pt>[r]_{\tau} \ar@<3pt>[r]^{\sigma} & 2
        \ar@<-3pt>[r]_{\tau} \ar@<3pt>[r]^{\sigma} & \cdots}
    \end{equation*}
    subject to the coglobular relations $\sigma \sigma = \sigma \tau$
    and $\tau \sigma = \tau \tau$. This means that for each $m > n$,
    there are precisely two maps
    $\sigma^{m-n}, \tau^{m-n} \colon n \rightrightarrows m$, which by
    abuse of notation we will write simply as $\sigma$ and $\tau$.

    The category $\E = [\mathbb G^\mathrm{op},\Set]$ is the category
    of \emph{globular sets}; it has a dense subcategory
    $\A = \Theta_0$, first described by Berger
    \cite{Berger2002A-Cellular}, whose objects have been termed
    \emph{globular cardinals} by Street \cite{Street2000The-petit}.
    The globular cardinals include the representables---the $n$-globes
    $Yn$ for each $n$---but also shapes such as the globular set with
    distinct cells as depicted below.
    \begin{equation}\label{eq:1}
      \xy
      (0,0)*+{\bullet}="00"; (20,0)*+{\bullet}="10";(40,0)*+{\bullet}="20";
      {\ar^{} "00"; "10"};{\ar@/^1.5pc/^{} "10"; "20"};{\ar@/_1.5pc/^{} "10"; "20"}; {\ar^{} "10"; "20"};{\ar@{=>}^{}(30,6)*+{};(30,1)*+{}};{\ar@{=>}^{}(30,-1)*+{};(30,-6)*+{}};
      \endxy
    \end{equation}
    The globular cardinals can be parametrised in various ways, for
    instance using trees \cite{Batanin1998Monoidal,
      Berger2002A-Cellular};
    following~\cite{Maltsiniotis2010Grothendieck}, we will use
    \emph{tables of dimensions}---sequences
    $\vec n=(n_{1}, \ldots ,n_{k})$ of natural numbers of odd length
    with $n_{2i-1} > n_{2i} < n_{2i+1}$. Given such a table $\vec{n}$
    and a functor $D \colon \mathbb G \to \C$, we obtain a diagram
    \begin{equation*}
      \xy
      (12,0)*+{Dn_{2}}="10"; (36,0)*+{Dn_{4}}="30";(60,-5)*+{\ldots}; (84,0)*+{Dn_{k-1}}="40";
      (0,-10)*+{Dn_{1}}="01"; (24,-10)*+{Dn_{3}}="21";(48,-10)*+{Dn_{5}}="41";(72,-10)*+{Dn_{k-2}}="31";(96,-10)*+{Dn_{k}}="51";
      {\ar_{D\tau} "10"; "01"};{\ar^{D\sigma} "10"; "21"};{\ar_{D\tau} "30"; "21"};{\ar^{D\sigma} "30"; "41"};{\ar_{D\tau} "40"; "31"};{\ar^{D\sigma} "40"; "51"};
      \endxy
    \end{equation*}
    whose colimit in $\C$, when it exists, will be written as
    $D(\vec n)$, and called the \emph{$D$-globular sum indexed by
      $\vec n$}. Taking
    $D=Y \colon \mathbb G \to [\mathbb G^\mathrm{op},\Set]$, the
    category $\Theta_0$ of globular cardinals is now defined as the
    full subcategory of $[\mathbb{G}^\mathrm{op}, \cat{Set}]$ spanned
    by the $Y$-globular sums. For example, the globular cardinal
    in~\eqref{eq:1} corresponds to the $Y$-globular sum
    $Y(1,0,2,1,2)$.

    This example captures algebraic structures on globular sets in
    which the operations build globes out of diagrams with shapes
    like~\eqref{eq:1}; these include \emph{strict $\omega$-categories}
    and \emph{strict $\omega$-groupoids}, but also the (globular)
    \emph{weak $\omega$-categories} and \emph{weak $\omega$-groupoids}
    studied in~\cite{Batanin1998Monoidal,Leinster2004Operads,
      Ara2013On-the-homotopy}.
    \label{item:2}
  \end{enumerate}
  We now turn to examples over enriched bases.
   
  \begin{enumerate}[(i)]
    \addtocounter{enumi}{5}
  \item Let $\V$ be a locally finitely presentable symmetric monoidal
    category whose finitely presentable objects are closed under the
    tensor product (cf.~\cite{Kelly1982Structures}). By taking
    $\E = \V$ and $\A = \V_{f}$ a skeleton of the full
    sub-$\V$-category of finitely presentable objects, we capture
    \emph{$\V$-enriched} finitary algebraic structure on $\V$-objects
    as studied in~\cite{Power1999Enriched}. When $\V = \cat{Cat}$ this
    means structure on categories $\C$ built from functors and natural
    transformations $\C^\I \rightarrow \C$ for finitely presentable
    $\I$: which includes \emph{symmetric monoidal} or \emph{finite
      limit} structure, but not \emph{symmetric monoidal closed} or
    \emph{factorization system} structure. Similarly, when
    $\V = \cat{Ab}$, it includes \emph{$A$-module} structure but not
    \emph{commutative ring} structure.
    \label{item:6}

  \item Taking $\V$ as before, taking $\E$ to be any locally finitely
    presentable $\V$-category~\cite{Kelly1982Structures} and taking
    $\A = \E_f$ a skeleton of the full subcategory of finitely
    presentable objects in $\E$, we capture $\V$-enriched finitary
    algebraic structure on $\E$-objects as studied
    in~\cite{Nishizawa2009Lawvere}. As before, there is the obvious
    generalization from finitary to $\lambda$-ary
    structure.\label{item:13}

  \item This example builds on~\cite{Lack2011Notions}. Let $\V$ be a
    locally presentable symmetric monoidal closed category, and
    consider a class of $\V$-enriched limit-types $\Phi$ with the
    property that the free $\Phi$-completion of a small $\V$-category
    is again small. A $\V$-functor $F \colon \C \rightarrow \V$ with
    small domain is called \emph{$\Phi$-flat} if its cocontinuous
    extension
    $\mathrm{Lan}_y F \colon [\C^\mathrm{op}, \V] \rightarrow \V$
    preserves $\Phi$-limits, and $a \in \V$ is
    \emph{$\Phi$-presentable} if
    $\V(a, \thg) \colon \V \rightarrow \V$ preserves colimits by
    $\Phi$-flat weights.

    Suppose that \emph{if $\C$ is small and $\Phi$-complete, then
      every $\Phi$-continuous $F \colon \C \rightarrow \V$ is
      $\Phi$-flat}; this is Axiom~A of~\cite{Lack2011Notions}. Then by
    Proposition~3.4 and \sec7.1 of \emph{ibid.}, we obtain an instance
    of our setting on taking $\E = \V$ and $\A = \V_\Phi$ a skeleton
    of the full sub-$\V$-category of $\Phi$-presentable objects.

    A key example takes $\V = \E = \Cat$ and $\Phi$ the class of
    finite products; whereupon $\V_\Phi$ is the subcategory
    $\mathbb{F}$ of finite cardinals, seen as discrete categories.
    This example captures \emph{strongly
      finitary}~\cite{Kelly1993Finite-product-preserving} structure on
    categories involving functors and transformations
    $\C^n \rightarrow \C$; this includes \emph{monoidal} or
    \emph{finite product} structure, but not \emph{finite limit}
    structure.
    \label{item:4}

  \item More generally, we can take
    $\E = \Phi\text-\cat{Cts}(\C, \V)$, the $\V$-category of
    $\Phi$-continuous functors $\C \rightarrow \V$ for some small
    $\Phi$-complete $\C$, and take $\A$ to be the full image of the
    Yoneda embedding
    $Y \colon \C^\mathrm{op} \rightarrow \Phi\text-\cat{Cts}(\C, \V)$.
    This example is appropriate to the study of ``$\Phi$-ary algebraic
    structure on $\E$-objects''---subsuming most of the preceding
    examples.\label{item:10}
  \end{enumerate}
\end{Exs}

\subsection{Pretheories as presentations}
\label{sec:preth-as-pres-1}

We will now describe examples of pretheories and their models in
various contexts; in doing so, it will be useful to avail ourselves of
the following constructions. Given a pretheory $\A \rightarrow \T$ and
objects $a,b \in \T$, to \emph{adjoin a morphism
  $f \colon a \rightarrow b$} is to form the $\V$-category $\T[f]$ in
the pushout square to the left of:
\begin{equation}\label{eq:4}
  \cd[@-0.3em]{
    {2} \ar[r]^-{\spn{a,b}} \ar[d]_{\iota} &
    {\T} \ar[d]^{\bar \iota} & &
    {\atwo +_2 \atwo} \ar[r]^-{\spn{f,g}} \ar[d]_{\spn{\id,\id}} &
    {\T} \ar[d]^{\bar \iota} \\
    {\atwo} \ar[r]^-{f} & \pushoutcorner
    {\T[f]} & &
    {\atwo} \ar[r]^-{f = g} & 
    \sh{r}{1em}{\T[f\!=\!g]\rlap{ .}} \pushoutcorner
  }
\end{equation}
Here, $\iota \colon 2 \rightarrow \atwo$ is the inclusion of the free
$\V$-category on the set $\{0,1\}$ into the free $\V$-category
$\atwo =\{0 \to 1\}$ on an arrow. Since $\iota$ is
identity-on-objects, its pushout $\bar \iota$ may also be chosen thus,
so that we may speak of adjoining an arrow to a pretheory
$J \colon \A \rightarrow \T$ to obtain the pretheory
$J[f] = \bar \iota \circ J \colon \A \rightarrow \T[f]$.

Recall from~\eqref{eq:15} that a concrete $\T$-model comprises
$X \in \E$ and $F \in [\T^\mathrm{op}, \V]$ for which
$F \circ J^\mathrm{op} = \E(K\thg, X) \colon \A \rightarrow \V$. Thus,
by the universal property of the pushout~\eqref{eq:4}, a concrete
$\T[f]$-model is the same as a concrete $\T$-model $(X,F)$ together
with a map $[f] \colon \E(Kb, X) \rightarrow \E(Ka,X)$ in $\V$.

Similarly given parallel morphisms $f,g \colon a \rightrightarrows b$
in the underlying category of $\T$ we can form the pushout above
right. In this way we may speak of adjoining an equation $f=g$ to a
pretheory $J \colon \A \rightarrow \T$ to obtain the pretheory
$J[f\!=\!g] = \bar \iota \circ J \colon \A \rightarrow \T[f\!=\!g]$.
In this case, we see that a concrete $\T[f\!=\!g]$-model is a concrete
$\T$-model $(X,F)$ such that
$Ff = Fg \colon \E(Kb,X) \rightarrow \E(Ka,X)$.

\begin{Ex}
  \label{ex:1}
  In the context of Examples~\ref{ex:2}\ref{item:3} appropriate to
  classical finitary algebraic theories--- so $\V = \E = \cat{Set}$
  and $\A = \mathbb{F}$---we will construct a pretheory
  $J \colon \mathbb F \to \M$ whose category of concrete models is the
  category of monoids.

  We start from the initial pretheory
  $\id \colon \mathbb{F} \rightarrow \mathbb{F}$ whose concrete models
  are simply sets, and construct from it a pretheory
  $J_1 \colon \mathbb{F} \rightarrow \M_1$ by adjoining morphisms
  \begin{equation}
    \label{eq:ops1}
    m \colon 1 \rightarrow 2 \qquad \text{and} \qquad i \colon 1 \rightarrow 0
  \end{equation}
  representing the monoid multiplication and unit operations, and also
  morphisms
  \begin{equation}\label{eq:ops2}
    m1,1m \colon 2 \rightrightarrows 3 \qquad \text{and} \qquad 
    i1,1i \colon 2 \rightrightarrows 1
  \end{equation}
  which will be necessary later to express the monoid equations. Note
  that our directional conventions mean that the input arity of these
  operations is in the \emph{codomain} rather than the domain. It
  follows from the preceding remarks that a concrete $\M_1$-model is a
  set $X$ equipped with functions
  \begin{equation*}
    [m] \colon X^2 \rightarrow X \text{ ,}\quad [i] \colon 1 \rightarrow X\text{ ,} \quad [m1],
    [1m] \colon X^3 \rightrightarrows X^2\text{ ,} \quad [i1], [1i] \colon 1 \rightrightarrows X 
  \end{equation*}
  interpreting the morphisms adjoined above. We now adjoin to $\M_1$
  the eight equations necessary to render commutative the following
  squares in $\M_1$:
  \begin{equation}\label{eq:ax1}
    \cd[@C-0.89em@R-0.4em]{
      1 \ar[r]^-{m} \ar[d]_-{\iota_1} & 2 \ar[d]^-{\iota_1} &
      1 \ar[r]^-{\id} \ar[d]_-{\iota_1} & 1 \ar[d]^-{\iota_1} &
      1 \ar[r]^-{i} \ar[d]_-{\iota_1} & 0 \ar[d]^-{!} &
      1 \ar[r]^-{\id} \ar[d]_-{\iota_1} & 1 \ar[d]^-{\iota_1} 
      \\
      1+1 \ar[r]^-{m1} & 2+1 &
      1+1 \ar[r]^-{1m} & 1+2 &
      1+1 \ar[r]^-{i1} & 1 &
      1+1 \ar[r]^-{1i} & 1 \\
      1 \ar[u]^-{\iota_2} \ar[r]^-{\id} & 1 \ar[u]_-{\iota_2} &
      1 \ar[u]^-{\iota_2} \ar[r]^-{m} & 2 \ar[u]_-{\iota_2} &
      1 \ar[u]^-{\iota_2} \ar[r]^-{\id} & 1 \ar[u]_-{\id} &
      1 \ar[u]^-{\iota_2} \ar[r]^-{i} & 0 \ar[u]_-{!} 
    }
  \end{equation}
  where $\iota_1$, $\iota_2$ and $!$ are the images under $J_1$ of the
  relevant coproduct injections or maps from $0$ in $\mathbb{F}$;
  together with three equations which render commutative:
  \begin{equation}\label{eq:ax2}
    \cd[@-0.1em]{
      1 \ar[d]_{m} \ar[r]^-{m} & 2 \ar[d]^{1m} & \ar[dr]_{1} 1 \ar[r]^{m} & 2 \ar[d]^{i1} & 1 \ar[dr]_-{1} \ar[r]^-{m}& 2 \ar[d]^-{1i}\\
      2 \ar[r]^-{m1} & 3 && 1 && 1\rlap{ .}
    }
  \end{equation}
  A concrete model for the resulting theory
  $J \colon \mathbb{F} \rightarrow \M$ is a concrete $\M_1$-model
  $(X, F)$ for which
  $F^\mathrm{op} \colon \M_1 \rightarrow \cat{Set}^\mathrm{op}$ sends
  each diagram in~\eqref{eq:ax1} and~\eqref{eq:ax2} to a commuting one.
  Commutativity in~\eqref{eq:ax1} forces
  $[m1] = [m] \times \id \colon X^3 \rightarrow X^2$ and so on;
  whereupon commutativity of~\eqref{eq:ax2} expresses precisely the
  monoid axioms, so that concrete $\M$-models are monoids, as desired.
  Extending this analysis to morphisms we see that $\conc[\M]$ is
  isomorphic to the category of monoids and monoid homomorphisms.
\end{Ex}

\begin{Ex}
  \label{ex:6}
  In the same way we can describe $\mathbb F$-pretheories modelling
  any of the categories of classical universal algebra---groups, rings
  and so on. Note that the same structure can be presented by distinct
  pretheories: for instance, we could extend the pretheory $\M$ of the
  preceding example by adjoining a further morphism
  $m11 \colon 3 \to 4$ and two equations forcing it to become
  $[m]\times 1 \times 1 \colon X^{4} \to X^{3}$ in any model; on doing
  so, we would not change the category of concrete models. However, in
  $\M$, all of the maps $3 \to 4$ belong to $\mathbb F$ while in the
  new pretheory, $m11$ does not. This non-canonicity will be rectified
  by the \emph{theories} introduced in
  Section~\ref{sec:monad-theory-corr} below; in particular,
  Corollary~\ref{cor:5} implies that, to within isomorphism, there is
  at most one $\mathbb F$-theory which captures a given type of
  structure.
\end{Ex}

\begin{Ex}
  \label{ex:3}
  In the situation of Examples~\ref{ex:2}\ref{item:1}, where
  $\E = [\mathbb{G}_0^\mathrm{op}, \cat{Set}]$ is the category of
  directed graphs and $\A = \Delta_0$, we will describe a pretheory
  $\Delta_0 \rightarrow \C$ whose concrete models are categories. The
  construction is largely identical to the example of monoids above.
  Starting from the initial $\Delta_0$-pretheory, we adjoin
  composition and unit maps $m \colon [1] \rightarrow [2]$ and
  $i \colon [1] \rightarrow [0]$ as well as the morphisms
  $1m,m1 \colon [2] \rightrightarrows [3]$ and
  $i1,1i \colon [2] \rightrightarrows [1]$ required to describe the
  category axioms.

  We now adjoin the necessary equations. First, we have four equations
  ensuring that composition and identities interact appropriately with
  source and target:
  \begin{equation*}
    \cd[@-0.5em]{
      {[0]} \ar[r]^-{\sigma} \ar[d]_{\sigma} &
      {[1]} \ar[d]^{m} &
      {[0]} \ar[r]^-{\tau} \ar[d]_{\tau} &
      {[1]} \ar[d]^{m} &
      [0] \ar[r]^-{\sigma} \ar[dr]_-{\id} & [1] \ar[d]^-{i} &
      [0] \ar[r]^-{\tau} \ar[dr]_-{\id} & [1] \ar[d]^-{i} \\
      {[1]} \ar[r]^-{\iota_1} &
      {[2]} &
      {[1]} \ar[r]^-{\iota_2} &
      {[2]} &
      & [0] & & [0]
    }
  \end{equation*}
  where here we write $\sigma, \tau \colon [0] \rightrightarrows [1]$
  for the two endpoint inclusions, and $\iota_1, \iota_2$ for the two
  colimit injections into $[1] \mathbin{{}_\tau +_\sigma} [1] = [2]$.
  We also require analogues of the eight equations of \eqref{eq:ax1}
  and three equations of \eqref{eq:ax2}. The modifications are minor:
  replace $n$ by $[n]$, the coproduct inclusions
  $\iota_1 \colon n \rightarrow n+m \leftarrow m \colon \iota_2$ by the
  pushout inclusions
  $\iota_1 \colon [n] \rightarrow [n] \mathbin{{}_\tau +_\sigma} [m]
  \leftarrow [m] \colon \iota_2$, the first appearance of
  $! \colon 0 \to 1$ by $\sigma \colon [0] \to [1]$ and its second
  appearance by $\tau \colon [0] \to [1]$. After adjoining these six
  morphisms and fifteen equations, we find that the concrete models of
  the resulting pretheory $\Delta_0 \rightarrow \C$ are precisely
  \emph{small categories}.

  We can extend this pretheory to one for groupoids. To do so, we
  adjoin a morphism $c \colon [1] \to [1]$ modelling the inversion plus
  the further maps $1c \colon [2] \to [2]$ and $c1 \colon [2] \to [2]$
  required for the axioms. Now four equations must be adjoined to force
  the correct interpretation of $1c$ and $c1$, plus the two equations
  for left and right inverses. On doing so, the resulting pretheory
  ${\Delta_0} \to \G$ has as its concrete models the \emph{small
    groupoids}.
\end{Ex}

\begin{Ex}
  \label{ex:4}
  In the situation of Examples~\ref{ex:2}\ref{item:2}, where $\E$ is
  the category of globular sets and $\A = \Theta_0$ is the full
  subcategory of globular cardinals, one can similarly construct
  pretheories whose concrete models are strict $\omega$-categories or
  strict $\omega$-groupoids. For instance, one encodes binary
  composition of $n$-cells along a $k$-cell boundary (for $k<n$) by
  adjoining morphisms $m_{n,k} \colon Y(n) \rightarrow Y(n,k,n)$ to
  ${\Theta_0}$. In fact, all of the standard flavours of globular
  \emph{weak} $\omega$-category and \emph{weak} $\omega$-groupoid can
  also be encoded using $\Theta_0$-pretheories; see
  Examples~\ref{exs:TheoriesDensity}\ref{item:14} below.
\end{Ex}

\begin{Ex}
  \label{ex:5}
  Consider the case of Examples~\ref{ex:2}\ref{item:4} where
  $\V = \E = \cat{Cat}$ and $\A = \mathbb{F}$, the full subcategory of
  finite cardinals (seen as discrete categories). We will describe an
  $\mathbb{F}$-pretheory capturing the structure of a monoidal
  category. In doing so, we exploit the fact that our pretheories are
  no longer mere categories, but $2$-categories; so we may speak not
  only of adjoining morphisms and equations between such, but also of
  \textit{adjoining an (invertible) $2$-cell}---by taking a pushout of
  the inclusion $\atwo +_2 \atwo \rightarrow D_2$ of the parallel pair
  $2$-category into the free $2$-category on an (invertible)
  $2$-cell---and similarly of \emph{adjoining an equation between
    $2$-cells}.

  To construct a pretheory for monoidal categories, we start
  essentially as for monoids: freely adjoining the usual maps
  $m, i, m1,1m,i1,1i$ to the initial pretheory, but now also morphisms
  $m11,1m1,11m \colon 3 \to 4$ and $1i1 \colon 3 \to 2$ needed for the
  monoidal category coherence axioms; thus, ten morphisms in all.

  We now add the $8 \times 2 = 16$ equations asserting that each of
  the morphisms beyond $m$ and $i$ has the expected interpretation in
  a model, plus\footnote{It may be prima facie unclear why this is
    necessary; after all, if $1m, m11, m1$ and $11m$ have the intended
    interpretations in a model, then it is certainly the case that
    they will verify this equality. Yet this equality is not forced to
    hold \emph{in the pretheory}, and we need it to do so in order
    for~\eqref{eq:5} to type-check.} the equation
  $1m\circ m11 = m1 \circ 11m \colon 2 \to 4$. This being done, we
  next adjoin invertible $2$-cells
  \begin{equation*}
    \cd{
      {1} \ar[r]^-{m} \ar[d]_{m} \dtwocell{dr}{\alpha} &
      {2} \ar[d]^{m1} \\
      {2} \ar[r]_-{1m} &
      {3}
    } \qquad
    \cd{
      \dtwocell[0.55]{dr}{\lambda} & 2 \ar[d]^-{i1} \\
      1 \ar[r]_-{1} \ar@/^1em/[ur]^{m} & 1
    } \qquad
    \cd{
      \utwocell[0.55]{dr}{\rho} & 2 \ar[d]^-{1i} \\
      1 \ar[r]_-{1} \ar@/^1em/[ur]^{m} & 1
    }
  \end{equation*}
  expressing the associativity and unit coherences, as well as the
  invertible $2$-cells
  \begin{equation*}
    \cd{
      {2} \ar[r]^-{m1} \ar[d]_{m1} \dtwocell{dr}{\alpha 1} &
      {3} \ar[d]^{m11} \\
      {3} \ar[r]_-{1m1} &
      {4}
    } \qquad
    \cd{
      {2} \ar[r]^-{1m} \ar[d]_{1m} \dtwocell{dr}{1\alpha} &
      {3} \ar[d]^{1m1} \\
      {3} \ar[r]_-{11m} &
      {4}
    } \qquad
    \cd{
      \dtwocell[0.55]{dr}{1\lambda} & 3 \ar[d]^-{1i1} \\
      2 \ar[r]_-{1} \ar@/^1em/[ur]^{1m} & 2
    } \qquad
    \cd{
      \utwocell[0.55]{dr}{\rho 1} & 3 \ar[d]^-{1i1} \\
      2 \ar[r]_-{1} \ar@/^1em/[ur]^{m1} & 2
    }
  \end{equation*}
  which will be needed to express the coherence axioms. Finally, we
  must adjoin equations between $2$-cells: the $2 \times 4 = 8$
  equations ensuring that $\alpha 1$, $1\alpha$, $1 \lambda$ and
  $\rho 1$ have the intended interpretation in any model, plus two
  equations expressing the coherence axioms:
  \begin{equation}\label{eq:5}
    \cd[@C+0.2em@R-0.2em@-0.2em]{
      & 2 \ar[r]^-{m1} \dtwocell{d}{\alpha} & 3 \ar[dr]^-{m11} \\
      1 \ar[ur]^-{m} \ar[r]|-{m} \ar[dr]_-{m} &
      2 \ar[ur]_-{1m} \ar[dr]^-{m1} \dtwocell{d}{\alpha} & & 4 \\
      & 2 \ar[r]_-{1m} & 3 \ar[ur]_-{11m}
    } \quad = \quad
    \cd[@C+0.2em@R-0.2em@-0.2em]{
      & 2 \ar[r]^-{m1} \ar[dr]|-{m1} \dtwocell{dd}{\alpha} & 3 \ar[dr]^-{m11}
      \dtwocell{d}{\alpha 1} \\
      1 \ar[ur]^-{m} \ar[dr]_-{m} & &
      3 \ar[r]|-{1m1} \dtwocell{d}{1 \alpha} & 4 \\
      & 2 \ar[r]_-{1m} \ar[ur]|-{1m} & 3 \ar[ur]_-{11m}
    }
  \end{equation}
  \begin{equation}
    \label{eq:8}
    \cd[@-0.2em]{
      & 2 \ar@/^1pc/[drr]^-{\id} \dtwocell{dd}{\alpha} \ar[dr]|-{m1} &
      {} \dtwocell[0.55]{d}{\rho 1} \\
      1 \ar[ur]^-{m} \ar[dr]_-{m} & & 3 \ar[r]|-{1i1}
      \dtwocell[0.45]{d}{1 \lambda} & 2 \\
      & 2 \ar@/_1pc/[urr]_-{\id} \ar[ur]|-{1m} & {}
    } \quad = \quad
    \cd[@-0.2em]{
      & 2 \ar@/^1pc/[drr]^-{\id} &
      {} \\
      1 \ar[ur]^-{m} \ar[dr]_-{m} \dtwocell{rrr}{\id} & & & 2\rlap{ .} \\
      & 2 \ar@/_1pc/[urr]_-{\id} & {}
    }
  \end{equation}
  All told, we have adjoined ten morphisms, seventeen equations between
  morphisms, seven invertible $2$-cells, and nine equations between
  $2$-cells to obtain a pretheory
  $J \colon \mathbb{F} \rightarrow \M\C$ whose concrete models are
  precisely monoidal categories.
\end{Ex}

\section{The monad--theory correspondence}
\label{sec:monad-theory-corr}

In this section, we return to the general theory and establish our
``best possible'' monad--theory correspondence. This will be obtained
by restricting the adjunction~\eqref{eq:adj} to its \emph{fixpoints}:
the objects on the left and right at which the counit and the unit are
invertible. The categories of fixpoints are the \emph{largest}
subcategories on which the adjunction becomes an adjoint equivalence,
and it is in this sense that our monad--theory correspondence is the
best possible.

\subsection{A pullback lemma}
\label{sec:pullback-lemma}

The following lemma will be crucial in characterising the fixpoints
of~\eqref{eq:adj} on each side. Note that the force of (2) below is in
the ``if'' direction; the ``only if'' is always true.
\begin{Lemma}
  \label{lem:2}
  A commuting square in $\V\text-\cat{CAT}$
  \begin{equation*}
    \cd{
      {\A} \ar[r]^-{F} \ar[d]_{H} &
      {\B} \ar[d]^{K} \\
      {\C} \ar[r]^-{G} &
      {\D}
    }
  \end{equation*}
  with $G$ fully faithful and $H,K$ discrete isofibrations is a
  pullback just when:
  \begin{enumerate}[itemsep=0em]
  \item $F$ is fully faithful; and
  \item An object $b \in \B$ is in the essential image of $F$ if and
    only if $Kb$ is in the essential image of $G$.
  \end{enumerate}
\end{Lemma}
\begin{proof}
  If the square is a pullback, then $F$ is fully faithful as a
  pullback of $G$. As for (2), if $Kb \cong Gc$ in $\D$ then since $K$
  is an isofibration we can find $b \cong b'$ in $\B$ with $Kb' = Gc$;
  now by the pullback property we induce $a \in \A$ with $Fa = b'$ so
  that $b \cong Fa$ as required. Suppose conversely that (1) and (2)
  hold. We form the pullback $\P$ of $K$ along $G$ and the induced map
  $L$ as below.
  \begin{equation}\label{eq:3}
    \cd[@-1.4em]{
      \A \ar@/^1em/[drrr]^-{F} \ar@/_1em/[dddr]_-{H}
      \ar@{-->}[dr]^-{L} \\ &
      \P \pullbackcorner \ar[rr]^-{P} \ar[dd]_-{Q} & & \B \ar[dd]^-{K}
      \\ \\
      & \C \ar[rr]^-{G} & & \D
    }
  \end{equation}
  $P$ is fully faithful as a pullback of $G$, and $F$ is so by
  assumption; whence by standard cancellativity properties of fully
  faithful functors, $L$ is also fully faithful.

  In fact, discrete isofibrations are also stable under pullback, and
  also have the same cancellativity property; this follows from the
  fact that they are the exactly the maps with the unique right lifting
  property against the inclusion of the free $\V$-category on an object
  into the free $\V$-category on an isomorphism. Consequently,
  in~\eqref{eq:3}, $Q$ is a discrete isofibration as a pullback of $K$,
  and $H$ is so by assumption; whence by cancellativity, $L$ is also a
  discrete isofibration.

  If we can now show $L$ is also essentially surjective, we will be
  done: for then $L$ is a discrete isofibration and an equivalence,
  whence invertible. So let $(b,c) \in \P$. Since $Kb = Gc$, by (2) we
  have that $b$ is in the essential image of $F$. So there is
  $a \in \A$ and an isomorphism $\beta \colon b \cong Fa$. Now
  $K\beta \colon Gc = Kb \cong KFa = GHa$ so by full fidelity of $G$
  there is $\gamma \colon c \cong Ha$ with $G\gamma = K\beta$; and so
  we have $(\beta, \gamma) \colon (b,c) \cong La$ exhibiting $(b,c)$ as
  in the essential image of $L$, as required.
\end{proof}

\subsection{$\A$-theories}

We first use the pullback lemma to describe the fixpoints
of~\eqref{eq:adj} on the pretheory side.

\begin{Defn}
  \label{def:1}
  An $\A$-pretheory $J \colon \A \to \T$ is said to be an
  \emph{$\A$-theory} if each $\T(J \thg,a) \in [\A^\mathrm{op}, \V]$
  is a $K$-nerve. We write $\th$ for the full subcategory of $\preth$
  on the $\A$-theories.
\end{Defn}

In the language of Section~\ref{sec:generalised-models} below, a
pretheory $\T$ is an $\A$-theory just when each representable
$\T(\thg,a) \colon \T^\mathrm{op} \to \V$ is a (\emph{non-concrete})
$\T$-model. When $\V = \E = \cat{Set}$ and $\A = \mathbb{F}$, an
$\A$-pretheory is an $\A$-theory precisely when it is a Lawvere
theory; see Examples~\ref{exs:TheoriesDensity}\ref{item:8} below.

\begin{Thm}\label{thm:unit}
  An $\A$-pretheory $J \colon \A \rightarrow \T$ is an $\A$-theory if
  and only if the unit component
  $\eta_\T \colon \T \rightarrow \Phi \Psi \T$ of~\eqref{eq:adj} is
  invertible.
\end{Thm}
\begin{proof}
  The unit $\eta_\T \colon \T \rightarrow \Phi \Psi \T$ is obtained by
  starting with $\alpha_{0}=1 \colon \Psi \T \to \Psi \T$ and chasing
  through the bijections of Theorem~\ref{thm:adjunction} to obtain
  $\alpha_{6}=\eta_\T$. Doing this, we quickly arrive at $\alpha_{2}$
  equal to $P$, the projection in the depicted pullback square
  \begin{equation}\label{eq:twodiagrams}
    \cd[@C+0.4em]{
      \Alg{\mathrm{\Psi}\mathnormal{\T}} \ar[r]^-{P} \ar[d]_{U^{\Psi
          \T}} \pullbackcorner &
      [\T^\op, \V] \ar[d]^{[J^\mathrm{op},1]} & 
      \A^\mathrm{op} \ar[d]_{(J_{\Psi \T})^\mathrm{op}}
      \ar[rr]^{J^\mathrm{op}} &&
      \T^\mathrm{op} \ar[d]^{\alpha_{4}} \ar@{.>}[dl]^{{\alpha_5}^\mathrm{op}} \ar@{.>}[dll]_{ \alpha_{6}^\mathrm{op}} \\
      \E \ar[r]^-{N_K} &
      [\A^\op, \V] & 
      (\A_{\Psi \T})^\mathrm{op}
      \ar[r]_-{({K_{\Psi \T}})^\mathrm{op}} &
      (\Alg{\mathrm{\Psi}\mathnormal{\T}})^\mathrm{op} \ar[r]_-{Y} &
      \sh{l}{0.7em}[\Alg{\mathrm{\Psi}\mathnormal{\T}},\V]_{\mathrm{rep}}
    }
  \end{equation}
  defining $\Alg{\mathrm{\Psi}\mathnormal{\T}}$. Now
  $\alpha_{3} \colon \Alg{\mathrm{\Psi}\mathnormal{\T}} \to
  [\T^\mathrm{op}, \V]$ is obtained by lifting an isomorphism through
  $[J^\mathrm{op},1]$ and so we have $\alpha_{3} \cong P$. We obtain
  $\alpha_{4}$ by transposing $\alpha_{3}$ through the isomorphism
  $(\thg)^{t} \colon
  \VCat(\Alg{\mathrm{\Psi}\mathnormal{\T}},[\T^\mathrm{op},\V])_{\mathrm{pwr}}
  \cong
  \VCat(\T^\mathrm{op},[\Alg{\mathrm{\Psi}\mathnormal{\T}},\V]_\mathrm{rep})$
  displayed in~\eqref{eq:33}. The relationships between $\alpha_{4}$,
  $\alpha_{5}$ and the unit component $\eta_\T= \alpha_{6}$ are
  depicted in the commutative diagram above right.

  The identity-on-objects unit $\eta_\T = \alpha_{6}$ will be
  invertible just when it is fully faithful which, since $K_{\Psi \T}$
  is fully faithful, will be so just when $\alpha_{5}$ is fully
  faithful. Now, since
  $P \cong \alpha_3 = (\alpha_4)^{t} = (Y \circ
  \alpha_5^\mathrm{op})^{t} = N_{\alpha_5}$, and $P$ is fully faithful,
  as the pullback of the fully faithful $N_K$, it follows that
  $N_{\alpha_5} \colon \E^{\Psi \T} \rightarrow [\T^\mathrm{op}, \V]$
  is also fully faithful. As a consequence, $\alpha_5$ is fully
  faithful just when there exists a factorisation to within
  isomorphism:
  \begin{equation}\label{eq:38}
    Y \cong N_{\alpha_5} \circ G \colon \T \rightarrow \E^{\Psi \T}
    \rightarrow [\T^\mathrm{op}, \V]\rlap{ .}
  \end{equation}
  Indeed, in one direction, if $\alpha_5$ is fully faithful then the
  canonical natural transformation
  $Y \Rightarrow N_{\alpha_5} \circ \alpha_5$ is invertible. In the
  other, given a factorisation as displayed, $G$ is fully faithful
  since $N_{\alpha_5}$ and $Y$ are. Moreover we have isomorphisms
  \begin{equation*}
    \Alg{\mathrm{\Psi}\mathnormal{\T}}(\alpha_5 b, \thg) \cong
    [\mathnormal{\T}^\mathrm{op},\V](Yb,N_{\alpha_5}\thg) \cong
    [\mathnormal{\T}^\mathrm{op},\V](N_{\alpha_5}Gb,N_{\alpha_5}\thg) \cong \Alg{\mathrm{\Psi}\mathnormal{\T}}(Gb,\thg)
  \end{equation*}
  natural in $b$. So by Yoneda, $\alpha_5 \cong G$ and so $\alpha_5$ is
  fully faithful since $G$ is so.

  This shows that $\eta_T$ is invertible just when there is a
  factorisation~\eqref{eq:38}. Since $N_{\alpha_5}$ is fully faithful
  this in turn is equivalent to asking that each $Yb = \T(\thg, b)$
  lies in the essential image of $\alpha_5$, or equally in the
  essential image of the isomorphic~$P$. As the left square
  of~\eqref{eq:twodiagrams} is a pullback, Lemma~\ref{lem:2} asserts
  that this is, in turn, equivalent to each
  $[J^\mathrm{op}, 1](Yb) = \T(J\thg, b)$ being in the essential image
  of $N_K$; which is precisely the condition that $J$ is an
  $\A$-theory.
\end{proof}

\subsection{$\A$-nervous monads}
\label{sec:a-nervous-monads}

We now characterise the fixpoints on the monad side. In the following
definition, $\A_\mathsf{T}$, $J_\mathsf{T}$ and $K_\mathsf{T}$ are as
in~\eqref{eq:factorisations}.

\begin{Defn}
  \label{def:2}
  A $\V$-monad $\mathsf{T}$ on $\E$ is called $\A$-\emph{nervous} if
  \begin{enumerate}[(i),itemsep=0em]
  \item The fully faithful
    $K_\mathsf{T} \colon \A_{\mathsf{T}} \to \Alg{T}$ is dense;
  \item A presheaf $X \in [{\A_{\mathsf{T}}}^\mathrm{op}, \V]$ is a
    $K_\mathsf{T}$-nerve if and only if
    $X \circ {J_{\mathsf T}}^\mathrm{op}$ is a $K$-nerve.
  \end{enumerate}
  We write $\nerv$ for the full subcategory of $\mnd$ on the
  $\A$-nervous monads.
\end{Defn}

Note that the adjointness isomorphisms
$\Alg{T}(K_{\mathsf T}J_{\mathsf T} X,Y) = \Alg{T}(F^{\mathsf T}KX,Y)
\cong \E(KX,U^{\mathsf T}Y)$ for the adjunction
$F^{\mathsf{T}} \dashv U^{\mathsf{T}}$ give a pseudo-commutative
square
\begin{equation}\label{eq:canIso}
  \cd[@C-1em]{
    \Alg{T} \ar[rr]^-{N_{K_{\mathsf T}}} \ar[d]_{U^{\mathsf T}} \twocong{drr} &&
    [{\A_{\mathsf T}}^\mathrm{op},\V] \ar[d]^{[{J_{\mathsf T}}^\mathrm{op},1]} \\
    \E \ar[rr]^-{N_{K}} &&
    [\A^\mathrm{op},\V]\rlap{ ;}}
\end{equation}
as a result of which, $[J_{\mathsf T}^\mathrm{op},1]$ maps
$K_\mathsf{T}$-nerves to $K$-nerves. Thus the force of clause (ii) of
the preceding definition lies in the \emph{if} direction.

\begin{Thm}
  \label{thm:counit}
  The counit component
  $\varepsilon_{\mathsf{T}} \colon \Psi \Phi \mathsf T \rightarrow
  \mathsf{T}$ of~\eqref{eq:adj} at a monad $\mathsf{T}$ on $\E$ is
  invertible if and only if $\mathsf{T}$ is $\A$-nervous.
\end{Thm}
\begin{proof}
  $\varepsilon_{\mathsf T}$ is obtained by taking
  $\alpha_6 = 1 \colon J_\mathsf{T} \to J_\mathsf{T}$ and proceeding
  in reverse order through the series of six natural isomorphisms in
  the proof of Theorem~\ref{thm:adjunction}. Doing this, we quickly
  reach $\alpha_{3}= N_{K_{\mathsf T}}$. Then
  $\alpha_{2} \colon \Alg{T} \to [(\A_{\mathsf T})^\mathrm{op},\V]$ is
  obtained by lifting the natural isomorphism $\phi$ of
  \eqref{eq:canIso} through the discrete isofibration
  $[J_{\mathsf T}^\mathrm{op},1]$, yielding a commutative square as
  left below.
  \begin{equation}\label{eq:liftIso}
    \cd{
      \Alg{T} \ar[rr]^-{\alpha_{2}} \ar[d]_{U^{\mathsf T}} &&
      [(\A_{\mathsf T})^\mathrm{op},\V] \ar[d]^{[J_{\mathsf T}^\mathrm{op},1]} \\
      \E \ar[rr]^-{N_{K}} &&
      [\A^\mathrm{op},\V]}
    \hspace{1cm}
    \cd{ \Alg{T} \ar[drr]_{U^{\mathsf T}} \ar[rr]^-{\alpha_1} &&
      \E^{\Psi \Phi \mathsf T} \ar[d]^{U^{\Psi \Phi \mathsf T}} \\
      && \E \rlap{.}} 
  \end{equation}
  The map $\alpha_{1} \colon \Alg{T} \to \E^{\Psi \Phi \mathsf T}$ is
  the unique map to the pullback, and
  $\alpha_{0}=\varepsilon_{\mathsf T}$ the corresponding morphism of
  monads. It follows that $\varepsilon_{\mathsf T}$ is invertible if
  and only the square to the left of~\eqref{eq:liftIso} is a pullback.
  Both vertical legs are discrete isofibrations and $N_K$ is fully
  faithful, so by Lemma~\ref{lem:2} this happens just when, firstly,
  $\alpha_2$ is fully faithful, and, secondly,
  $X \in [\A_\mathsf{T}^\mathrm{op}, \V]$ is in the essential image of
  $\alpha_2$ if and only if $XJ_\mathsf{T}$ is a $K$-nerve. But as
  $\alpha_2 \cong N_{K_\mathsf{T}}$, and natural isomorphism does not
  change either full fidelity or essential images, this happens just
  when $\mathsf{T}$ is $\A$-nervous.
\end{proof}

\subsection{The monad--theory equivalence}
\label{sec:monad-theory-equiv}

Putting together the preceding results now yields the main result of
this paper.

\begin{Thm}
  \label{thm:1}
  The adjunction \eqref{eq:adj} restricts to an adjoint equivalence
  \begin{equation}\label{eq:eq}
    \cd[@C+1em]{
      {\nerv} \ar@<-4.5pt>[r]_-{\Phi} \ar@{}[r]|-{\bot} &
      {\th} \ar@<-4.5pt>[l]_-{\Psi}
    }
  \end{equation}
  between the category of $\A$-nervous monads and the category of
  $\A$-theories.
\end{Thm}
\begin{proof}
  Any adjunction restricts to an adjoint equivalence between the
  objects with invertible unit and counit components respectively, and
  Theorems~\ref{thm:unit} and~\ref{thm:counit} identify these objects
  as the $\A$-theories and the $\A$-nervous monads.
\end{proof}

Note that there is an asymmetry between the conditions found on each
side. On the one hand, the condition characterising the $\A$-theories
among the $\A$-pretheories is intrinsic, and easy to check in
practice. On the other hand, the condition defining an $\A$-nervous
monad refers to the associated pretheory, and is non-trivial to check
in practice. Indeed, one of the main points
of~\cite{Weber2007Familial, Berger2012Monads} is to provide a general
set of \emph{sufficient} conditions under which a monad can be shown
to be $\A$-nervous.

In the sections which follow, we will provide a number of more
tractable characterisations of the $\A$-theories and $\A$-nervous
monads; the crucial fact which drives all of these is that the
adjunction~\eqref{eq:adj} is in fact idempotent. Recall that an
adjunction $L \dashv R \colon \D \rightarrow \C$ is \emph{idempotent}
if the monad $RL$ on $\C$ is idempotent, and that this is equivalent
to asking that the comonad $LR$ is idempotent, or that any one of the
natural transformations $R\varepsilon$, $\varepsilon L$, $\eta R$ and
$L \eta$ is invertible.

\begin{Thm}\label{cor:idempotent}
  The adjunction~\eqref{eq:adj} is idempotent.
\end{Thm}
\begin{proof}
  We show for each $\mathsf{T} \in \mnd$ that the unit
  $\eta_{\Phi \mathsf{T}} \colon \Phi \mathsf{T} \to \Phi \Psi \Phi
  \mathsf{T}$ is invertible. By Theorem~\ref{thm:unit}, this is
  equally to show that $J_{\mathsf T} \colon \A \to \A_{\mathsf T}$ is
  an $\A$-theory, i.e., that each
  $\A_\mathsf{T}(J_\mathsf{T}\thg, J_\mathsf{T}a) \in [\A^\mathrm{op},
  \V]$ is a $K$-nerve. But
  $\A_{\mathsf T}(J_{\mathsf T}\thg, J_\mathsf{T}a) \cong
  \Alg{T}(F^{\mathsf T}K\thg ,F^{\mathsf T}Ka) \cong
  \E(K\thg,U^{\mathsf T}F^{\mathsf T}Ka) = \E(K\thg, TKa)$ as
  required.
\end{proof}
Exploiting the alternative characterisations of idempotent adjunctions
listed above, we immediately obtain the following result, which tells
us in particular that a monad $\mathsf{T}$ is $\A$-nervous if and only
if it can be presented by some $\A$-pretheory.
\begin{Cor}
  \label{cor:2}
  A monad $\mathsf{T}$ on $\E$ is $\A$-nervous if and only if
  $\mathsf{T} \cong \Psi \T$ for some $\A$-pretheory
  $J \colon \A \rightarrow \T$; while an $\A$-pretheory
  $J \colon \A \rightarrow \T$ is an $\A$-theory if and only if
  $\T \cong \Phi\mathsf{T}$ for some monad $\mathsf{T}$ on $\E$.
\end{Cor}
The next result also follows directly from the definition of
idempotent adjunction.
\begin{Cor}
  \label{cor:1}
  The full subcategory $\cat{Mnd}_{\A}(\E) \subseteq \mnd$ is
  coreflective via $\Psi\Phi$, while the full subcategory
  $\cat{Th}_{A}(\E) \subseteq \preth$ is reflective via $\Phi \Psi$.
\end{Cor}

\section{Semantics}\label{sec:semantics}

In the next section, we will explicitly identify the $\A$-nervous
monads and $\A$-theories for the examples listed in
Section~\ref{sec:setting}. Before doing this, we study further aspects
of the general theory, namely those related to the taking of
semantics.

\subsection{Interaction with the semantics functors}
\label{sec:inter-with-adjunct}

We begin by examining the interaction of our monad--theory
correspondence with the semantics functors of
Section~\ref{sec:monads-pretheories}. In fact, we begin at the level
of the monad--pretheory adjunction~\eqref{eq:adj}.

\begin{Prop}\label{prop:semtheory}
  There is a natural isomorphism $\theta$ as on the left in:
  \begin{equation*}
    \cd[@!C@C-3.5em]{
      {\preth^\mathrm{op}} \ar[rr]^{\Psi^\mathrm{op}}
      \ar[dr]_-{\mathrm{Mod}_c} &
      \ltwocello[0.5]{d}{\theta} &
      {\mnd^\mathrm{op}} \ar[dl]^-{\mathrm{Alg}} & & 
      {\mnd^\mathrm{op}} \ar[rr]^{\Phi^\mathrm{op}} \ar[dr]_-{\mathrm{Alg}} &
      \rtwocell[0.5]{d}{\bar\theta} &
      {\preth^\mathrm{op}} \ar[dl]^-{\mathrm{Mod}_c} \\ &
      {\V\text-\cat{CAT}/\E} & & & &
      {\V\text-\cat{CAT}/\E}\rlap{ .}
    }
  \end{equation*}
  Let $\bar\theta$ be its mate under the adjunction
  $\Phi^\mathrm{op} \dashv \Psi^\mathrm{op}$, as right above. The
  component of $\bar \theta$ at $\mathsf{T} \in \mnd$ is invertible if
  and only if $\mathsf{T}$ is $\A$-nervous.
\end{Prop}
\begin{proof}
  For the first claim, Theorem~\ref{thm:adjunction} provides the
  necessary natural isomorphisms
  $\theta_\T \colon \E^{\Psi \T} \rightarrow \conc$ over $\E$. For the
  second, if we write as before
  $\varepsilon_{\mathsf{T}} \colon \Psi \Phi \mathsf T \rightarrow
  \mathsf{T}$ for the counit component of~\eqref{eq:adj} at
  $\mathsf{T} \in \mnd$, then the $\mathsf{T}$-component of
  $\bar\theta$ is the composite
  $\theta_{\Psi \mathsf T} \circ (\varepsilon_{\mathsf{T}})^\ast
  \colon \E^{\mathsf{T}} \rightarrow \E^{\Psi\Phi\mathsf{T}}
  \rightarrow \conc[\Phi\T]$ over $\E$. Since $\theta_{\Psi \T}$ is
  invertible and since $\mathrm{Alg}$ is fully faithful,
  ${\bar\theta}_\mathsf{T}$ will be invertible just when
  $\varepsilon_\mathsf{T}$ is so; that is, by
  Theorem~\ref{thm:counit}, just when $\mathsf{T}$ is $\A$-nervous.
\end{proof}

From this and the fact that each monad $\Psi \T$ is $\A$-nervous, it
follows that \emph{an $\A$-pretheory $\T$ and its associated theory
  $\Phi\Psi \T$ have isomorphic categories of concrete models}. By
contrast, the passage from a monad $\mathsf{T}$ to its $\A$-nervous
coreflection $\Psi \Phi \mathsf{T}$ may well change the category of
algebras. For example, the power-set monad on $\cat{Set}$, whose
algebras are complete lattices, has its $\mathbb{F}$-nervous
coreflection given by the finite-power-set monad, whose algebras are
$\vee$-semilattices. However, if we restrict to $\A$-nervous monads
and $\A$-theories, then the situation is much better behaved.

\begin{Thm}
  \label{cor:5}
  The monad--theory equivalence~\eqref{eq:eq} commutes with the
  semantics functors; that is, we have natural isomorphisms:
  \begin{equation}\label{eq:34}
    \cd[@!C@C-3.5em]{
      {\th^\mathrm{op}} \ar[rr]^{\Psi^\mathrm{op}} \ar[dr]_-{\mathrm{Mod}_c} &
      \ltwocello[0.4]{d}{\theta} &
      {\nerv^\mathrm{op}} \ar[dl]^-{\mathrm{Alg}} & & 
      {\nerv^\mathrm{op}} \ar[rr]^{\Phi^\mathrm{op}} \ar[dr]_-{\mathrm{Alg}} &
      \rtwocell[0.4]{d}{\bar\theta} &
      {\th^\mathrm{op}} \ar[dl]^-{\mathrm{Mod}_c} \\ &
      {\V\text-\cat{CAT}/\E} & & & &
      {\V\text-\cat{CAT}/\E}\rlap{ .}
    }
  \end{equation}
  Moreover, both
  $\mathrm{Mod}_c\colon \th^\mathrm{op} \rightarrow
  \V\text-\cat{CAT}/\E$ and
  $\mathrm{Alg} \colon \nerv^\mathrm{op} \rightarrow \V\text-\cat{CAT}
  / \E$ are fully faithful functors.
\end{Thm}
\begin{proof}
  The first statement follows from Proposition~\ref{prop:semtheory}.
  For the second, note that
  $\mathrm{Alg} \colon {\nerv}^\mathrm{op} \to \VCat/\E$ is obtained
  by restricting the fully faithful
  $\mathrm{Alg} \colon {\mnd}^\mathrm{op} \to \VCat/\E$ along a full
  embedding, and so is itself fully faithful. It follows that
  $\mathrm{Mod}_c \cong \mathrm{Alg} \circ \Psi^\mathrm{op} \colon
  \th^\mathrm{op} \rightarrow \VCat / \E$ is also fully faithful.
\end{proof}

Full fidelity of
$\mathrm{Mod}_c \colon \th^\mathrm{op} \rightarrow \VCat/\E$ means
that an $\A$-theory is determined to within isomorphism by its
category of concrete models over $\E$. This rectifies the
non-uniqueness of pretheories noted in Example~\ref{ex:6} above.

\subsection{Non-concrete models}
\label{sec:generalised-models}

In Section~\ref{sec:pretheories} we defined a \emph{concrete model} of
an $\A$-pretheory $\T$ to be an object $X \in \E$ endowed with an
extension of $\E(K\thg, X) \colon \A^\mathrm{op} \rightarrow \V$ to a
functor $\T^\mathrm{op} \rightarrow \V$. In the literature, one often
encounters a looser notion of model for a theory, in which an
underlying object in $\E$ is not provided. In our setting, this notion
becomes the following one: by an (unqualified) \emph{$\T$-model}, we
mean a functor $F \colon \T^\mathrm{op} \rightarrow \V$ whose
restriction $FJ^\mathrm{op} \colon \A^\mathrm{op} \rightarrow \V$ is a
$K$-nerve.

The $\T$-models span a full sub-$\V$-category $\cat{Mod}(\T)$ of
$[\T^\mathrm{op}, \V]$. Recalling from Section~\ref{sec:setting} that
$\KN(\V)$ denotes the full sub-$\V$-category of $[\A^\mathrm{op}, \V]$
on the $K$-nerves, we may also express $\cat{Mod}(\T)$ as a pullback
as to the right in:
\begin{equation}\label{eq:9}
  \cd{
    {\conc} \ar@/^1.1em/[rr]^-{P_\T} \ar@{-->}[r]_-{} \ar[d]_{U_\T} &
    {\cat{Mod}(\T)} \ar@{ (->}[r] \ar[d]_{W_{\T}} &
    {[\T^\mathrm{op}, \V]} \ar[d]^{[J^\mathrm{op}, 1]} \\
    {\E} \ar[r]^-{N_K} &
    {\KN(\V)} \ar@{ (->}[r] &
    {[\A^\mathrm{op}, \V]}\rlap{ .}
  }
\end{equation}
On the other hand, $\conc$ is the pullback around the outside, and so
there is a canonical induced map $\conc \rightarrow \cat{Mod}(\T)$ as
displayed. By the usual cancellativity properties, the left square
above is now also a pullback. Moreover, $W_\T$ is an isofibration, as
a pullback of the discrete isofibration $[J^\mathrm{op}, 1]$, and
$N_K \colon \E \rightarrow \KN(\V)$ is an equivalence. Since
equivalences are stable under pullback along isofibrations, we
conclude that:
\begin{Prop}
  \label{prop:1}
  The comparison $\conc \rightarrow \cat{Mod}(\T)$
  in~\eqref{eq:9} is an equivalence.
\end{Prop}

Taking non-concrete models gives rise to a semantics functor landing
in $\VCat/\KN(\V)$ which, like before, is \emph{not} fully faithful on
$\A$-pretheories, but is so on the subcategory of $\A$-theories. Note
that the ``underlying $K$-nerve'' of a $\T$-model is more natural than
it might seem, being the special case of the functor
$\cat{Mod}(\T) \rightarrow \cat{Mod}(\S)$ induced by a morphism of
$\A$-pretheories for which $\S$ is the initial pretheory. However, in
the following result, for simplicity, we view the semantics functors
for $\T$-models as landing simply in $\VCat$.

\begin{Thm}
  \label{thm:9}
  The monad--theory equivalence~\eqref{eq:eq} commutes with the
  non-concrete semantics functors in the sense that we have natural
  transformations
  \begin{equation*}
    \cd[@!C@C-3.5em]{
      {\th^\mathrm{op}} \ar[rr]^{\Psi^\mathrm{op}} \ar[dr]_-{\mathrm{Mod}} &
      \ltwocello[0.4]{d}{\theta} &
      {\nerv^\mathrm{op}} \ar[dl]^-{\mathrm{Alg}} & & 
      {\nerv^\mathrm{op}} \ar[rr]^{\Phi^\mathrm{op}} \ar[dr]_-{\mathrm{Alg}} &
      \rtwocell[0.4]{d}{\bar\theta} &
      {\th^\mathrm{op}} \ar[dl]^-{\mathrm{Mod}} \\ &
      {\V\text-\cat{CAT}} & & & &
      {\V\text-\cat{CAT}}
    }
  \end{equation*}
  whose components are equivalences in $\VCat$.
\end{Thm}
\begin{proof}
  First postcompose the natural isomorphisms~\eqref{eq:34} with the
  forgetful functor $\VCat/\E \rightarrow \VCat$. Then paste with the
  natural transformation
  $\mathrm{Mod}_c \Rightarrow \mathrm{Mod} \colon \th^\mathrm{op}
  \rightarrow \VCat$ coming from the previous proposition.
\end{proof}

\subsection{Local presentability and algebraic left adjoints}

Next in this section, we consider the categorical properties of the
$\V$-categories and $\V$-functors in the image of the semantics
functors. We begin with the case of pretheories.

\begin{Prop}\label{prop:models1}
  \begin{enumerate}[(i)]
    \item If $J \colon \A \to \T$ is an $\A$-pretheory then $\conc$ is
      locally presentable and $U_{\T} \colon \conc \to \E$ is a strictly
      monadic right adjoint.
    \item If $H \colon \T \rightarrow \S$ is a map of $\A$-pretheories,
      then $H^\ast \colon \conc[\S] \rightarrow \conc$ is a strictly
      monadic right adjoint.
    \end{enumerate}
\end{Prop}
\begin{proof}
  (i) follows from Lemma~\ref{lem:1} and the description in
  \eqref{eq:15} of $\conc \to \E$ as a pullback. For (ii), applying
  the standard cancellativity properties to the pullbacks defining
  $\conc[\S]$ and $\conc$ yields a pullback square
  \begin{equation*}
    \cd{
      {\conc[\S]} \pullbackcorner \ar[r]^-{P_\S} \ar[d]_{H^\ast} &
      {[\S^\mathrm{op}, \V]} \ar[d]^{[H^\mathrm{op}, 1]} \\
      {\conc} \ar[r]^-{P_\T} &
      {[\T^\mathrm{op}, \V]}\rlap{ .}
    }
  \end{equation*}
  Since $[H^\mathrm{op}, 1]$ is strictly monadic and $P_\T$ is a right
  adjoint between locally presentable categories, the result follows
  again from Lemma~\ref{lem:1}.
\end{proof}
Composing with the equivalence $\conc \simeq \cat{Mod}(\T)$ of
Proposition~\ref{prop:1}, this result immediately implies the local
presentability of the category $\cat{Mod}(\T)$ of non-concrete models.
Likewise, in the non-concrete setting, the analogue of
Proposition~\ref{prop:models1} remains true on replacing ``strict
monadicity'' by ``monadicity'' throughout. On the other hand, taken
together with Proposition~\ref{prop:semtheory}, it immediately implies
the corresponding result for nervous monads. We state this here as:
\begin{Prop}
  \begin{enumerate}[(i)]
  \item If $\mathsf{T}$ is an $\A$-nervous monad then $\E^\mathsf{T}$
    is locally presentable, and
    $U^\mathsf{T} \colon \E^\mathsf{T} \to \E$ is a strictly monadic
    right adjoint.
  \item If $\alpha \colon \mathsf{T} \rightarrow \mathsf{S}$ is a map
    of $\A$-nervous monads, then
    $\alpha^\ast \colon \E^{\mathsf{S}} \rightarrow \E^{\mathsf{T}}$
    is a strictly monadic right adjoint.
  \end{enumerate}
\end{Prop}

\subsection{Algebraic colimits of monads and theories}\label{section:AlgebraicColimits}

To conclude this section, we examine the interaction of the semantics
functors with colimits. We begin with the more-or-less classical case
of the semantics functor for monads
$\mathrm{Alg} \colon \mnd^\mathrm{op} \to \VCat/\E$.

In general, $\mnd$ need not be cocomplete. Indeed, when
$\V = \E = \cat{Set}$, it does not even have all binary coproducts;
see~\cite[Proposition~6.10]{Barr1970Coequalizers}. However many
colimits of monads do exist, and an important point about these is
that, in the terminology of~\cite{Kelly1980A-unified}, they are
\emph{algebraic}. That is, they are sent to limits by the semantics
functor $\mathrm{Alg} \colon \mnd^\mathrm{op} \to \VCat/\E$.

To prove this, we use the following lemma, which is a mild variant of
the standard result that right adjoints preserve limits.

\begin{Lemma}\label{lemma:generator}
  Let $\C$ be a complete (ordinary) category with a strongly
  generating class of objects $X$ and consider a functor
  $U \colon \A \to \C$. If each $x \in X$ admits a reflection along
  $U$ then $U$ preserves any limits that exist in $\A$.
\end{Lemma}

\begin{proof}
  As $X$ is a strong generator, the functors $\C(x,\thg)$ with
  $x \in X$ jointly reflect isomorphisms, and so jointly reflect
  limits. Accordingly $U$ preserves any limits that are preserved by
  $\C(x,U\thg)$ for each $x \in X$. But each $\C(x,U\thg)$ is
  representable and so preserves all limits; whence $U$ preserves any
  limits that exist.
\end{proof}

In the setting of $\Set$-enriched categories the following result,
expressing the algebraicity of colimits of monads, is a special case
of Proposition~26.3 of \cite{Kelly1980A-unified}.

\begin{Prop}\label{prop:algcol2}
  $\mathrm{Alg} \colon \mnd^\mathrm{op} \to \VCat/\E$ preserves
  limits.
\end{Prop}

\begin{proof}
  The $\V$-functors $F \colon \X \to \E$ with small domain form a
  strong generator for $\VCat/\E$. Moreover, it is shown
  in~\cite[Theorem~II.1.1]{Dubuc1970Kan-extensions} that each such $F$
  has a reflection along
  $\mathrm{Alg} \colon \mnd^\mathrm{op} \rightarrow \VCat / \E$ given
  by its codensity monad $\mathrm{Ran}_F(F) \colon \E \rightarrow \E$.
  The result thus follows from Lemma~\ref{lemma:generator}.
\end{proof}

We now adapt the above results concerning $\mnd$ to the cases of
$\preth$, $\nerv$ and $\th$. In Theorem~\ref{thm:7} below, we will see
that these categories are locally presentable; in particular, and by
contrast with $\mnd$, they are cocomplete. It is also not difficult to
prove the cocompleteness directly.

\begin{Prop}\label{prop:algcol3}
  Each of the semantics functors
  $\mathrm{Alg} \colon \nerv^\mathrm{op} \to \VCat/\E$,
  $\mathrm{Mod}_c \colon \preth^\mathrm{op} \rightarrow \VCat/\E$ and
  $\mathrm{Mod}_c \colon \th^\mathrm{op} \rightarrow \VCat/\E$
  preserves limits.
\end{Prop}

\begin{proof}
  These three functors are isomorphic to the respective composites:
  \begin{align}
    \nerv^\mathrm{op} \xrightarrow{\mathrm{incl}^\mathrm{op}} & \mnd^\mathrm{op}
    \xrightarrow{\mathrm{Alg}} \VCat/\E \label{eq:11}\\[-1ex]
    \preth^\mathrm{op} \xrightarrow{\Psi^\mathrm{op}} & \mnd^\mathrm{op}
    \xrightarrow{\mathrm{Alg}} \VCat/\E\label{eq:22}\\[-1ex]
    \th^\mathrm{op} \xrightarrow{\Psi^\mathrm{op}} & \mnd^\mathrm{op}
    \xrightarrow{\mathrm{Alg}} \VCat/\E\rlap{ ;}\label{eq:23}
  \end{align}
  for~\eqref{eq:11} this is clear, while for~\eqref{eq:22}
  and~\eqref{eq:23} it follows from Proposition~\ref{prop:semtheory}.
  The common second functor in each composite is limit-preserving by
  Proposition~\ref{prop:algcol2}, while the first functor is
  limit-preserving in each case since it is the opposite of a left
  adjoint functor---by Corollary~\ref{cor:1},
  Theorem~\ref{thm:adjunction} and Theorem~\ref{thm:1} (taken together
  with Corollary~\ref{cor:1}) respectively.
\end{proof}

We leave it to the reader to formulate this result also for
non-concrete models.

\section{The monad--theory correspondence in practice}
\label{sec:identifying-theories}

In this section, we return to the examples of our general setting
described in Section~\ref{sec:setting}, with the goal of describing as
explicitly as possible what the $\A$-nervous monads, the
$\A$-theories, and the corresponding models look like in each case. By
way of these descriptions, we will re-find many of the monad--theory
correspondences existing in the literature as instances of our main
Theorem~\ref{thm:1}.

To obtain our explicit descriptions, we will require some further
results which characterise $\A$-theories and $\A$-nervous monads in
particular situations. We begin this section by describing these
results.

\subsection{Theories in the presheaf context}
\label{sec:presheaf-context}

A number of the examples of our basic setting described in
Section~\ref{sec:basic-setting} arise in the following manner. We take
$\E = [\C^\mathrm{op},\V]$ a presheaf category, and take $\A$ to be
any full subcategory of $\E$ containing the representables. In this
situation, we then have a factorisation
\begin{equation}\label{eq:28}
  \xymatrix{
    \C \ar[r]^-{I}
    & \A \ar[r]^-{K} & [\C^\mathrm{op},\V] = \E
  }
\end{equation}
of the Yoneda embedding. The Yoneda lemma implies that
$Y \colon \C \rightarrow [\C^\mathrm{op}, \V]$ is dense, whereupon by
Theorem~5.13 of \cite{Kelly1982Basic}, both $I$ and $K$ are too. In
particular, $K$ provides an instance of our basic setting; we will
call this the \emph{presheaf context}. Each of
Examples~\ref{ex:2}\ref{item:3}, \ref{item:1}, \ref{item:2},
\ref{item:6}, and \ref{item:4} arise in this way.

\begin{Lemma}\label{lem:5}
  In the presheaf context, we have $N_I \cong K$ and
  $N_{K} \cong \Ran_{I^\mathrm{op}}$. Moreover, a functor
  $F \colon \A^\mathrm{op} \to \V$ is a $K$-nerve just when it is the
  right Kan extension of its restriction along
  $I^\mathrm{op} \colon \C^\mathrm{op} \rightarrow \A^\mathrm{op}$.
\end{Lemma}

\begin{proof}
  For the first isomorphism we calculate that
  \begin{equation}\label{eq:29}
    N_I(x) = \A(I\thg, x) \cong [\C^\mathrm{op}, \V](KI\thg, Kx) =
    [\C^\mathrm{op}, \V](Y\thg, Kx) \cong Kx
  \end{equation}
  by full fidelity of $K$ and the Yoneda lemma. For the second, since
  $\mathrm{Lan}_Y K \dashv N_K$ and
  $[I^\mathrm{op}, 1] \dashv \mathrm{Ran}_{I^\mathrm{op}}$ it suffices
  to show
  $\mathrm{Lan}_Y K \cong [I^\mathrm{op}, 1] \colon [\A^\mathrm{op},
  \V] \rightarrow [\C^\mathrm{op}, \V]$. Since both are cocontinuous,
  it suffices to show $(\mathrm{Lan}_Y K)Y \cong [I^\mathrm{op}, 1]Y$,
  which follows since
  $(\mathrm{Lan}_Y K)Y \cong K \cong N_I = [I^\mathrm{op}, 1]Y$ using
  full fidelity of $Y$ and~\eqref{eq:29}. Finally, since
  $I^\mathrm{op}$ is fully faithful,
  $F \colon \A^\mathrm{op} \rightarrow \V$ is a right Kan extension
  along $I^\mathrm{op}$ just when it is the right Kan extension of its
  own restriction. Thus the final claim follows using the isomorphism
  $N_K \cong \mathrm{Ran}_{I^\mathrm{op}}$.
\end{proof}

In this setting, we have practically useful characterisations of the
$\A$-theories and their (non-concrete) models.

\begin{Prop}
  \label{prop:6}
  Let $J \colon \A \rightarrow \T$ be an $\A$-pretheory in the
  presheaf context~\eqref{eq:28}.
  \begin{enumerate}[(i),itemsep=0em]
  \item A functor $F \colon \T^\mathrm{op} \rightarrow \V$ is a
    $\T$-model just when
    $FJ^\mathrm{op} \colon \A^\mathrm{op} \rightarrow \V$ is the right
    Kan extension of its restriction along
    $I^\mathrm{op} \colon \C^\mathrm{op} \rightarrow \A^\mathrm{op}$;
  \item $J \colon \A \rightarrow \T$ is itself an $\A$-theory just
    when it is the pointwise left Kan extension of its restriction
    along $I \colon \C \rightarrow \A$.
  \end{enumerate}
\end{Prop}

\begin{proof}
  (i) follows immediately from Lemma~\ref{lem:5} since, by definition,
  $F$ is a $\T$-model just when $FJ^\mathrm{op}$ is a $K$-nerve. For
  (ii), note that by Proposition~4.46 of~\cite{Kelly1982Basic},
  $J \colon \A \rightarrow \T$ is the pointwise left Kan extension of
  its restriction along $I$ just when, for each $x \in \T$, the
  functor $\T(J\thg, x) \colon \A^\mathrm{op} \rightarrow \V$ is the
  right Kan extension of its restriction along $I^\mathrm{op}$. By
  Lemma~\ref{lem:5}, this happens just when each $\T(J\thg, x)$ is a
  $K$-nerve---that is, just when $J$ is a $\A$-theory.
\end{proof}

We can sharpen these results using Day's notion of \emph{density
  presentation}~\cite{Day1974On-closed}. The density of an ordinary
functor $K \colon \C \rightarrow \D$ is often introduced as the
assertion that each object of $\D$ is the colimit of a certain diagram
in the image of $K$. It is this perspective that the notion of density
presentation generalises.

A \emph{family of colimits} $\Phi$ in the ordinary category $\D$ is a
class of diagrams $(D_i \colon \J_i \rightarrow \D)_{i \in I}$ each of
which has a colimit in $\D$. In the enriched case, a \emph{family of
  colimits} $\Phi$ in the $\V$-category $\D$ is a class of pairs
${(W_i \in [\J_i^\mathrm{op}, \V], D_i \colon \J_i \rightarrow \D)_{i
    \in I}}$ such that each weighted colimit $W_i \star D_i$ exists in
$\D$. In either case, a full replete subcategory $\B$ of $\D$ is
\emph{closed in $\D$ under $\Phi$-colimits} if it contains the
(weighted) colimit of any $D_i$ in $\Phi$ whenever it contains each
vertex of $D_{i}$. We say that $\D$ is the \emph{closure of $\B$ under
  $\Phi$-colimits} if the only replete full subcategory of $\D$
containing $\B$ and closed under $\Phi$-colimits is $\D$ itself.

Now given a fully faithful $K \colon \C \rightarrow \D$, we say that a
colimit in $\D$ is \emph{$K$-absolute} if it is preserved by $N_K$, or
equivalently, by each representable
$\D(Kx, \thg) \colon \D \rightarrow \V$. If $\D$ is the closure of
$\C$ under a family $\Phi$ of $K$-absolute colimits then $\Phi$ is
said to be a \emph{density presentation} for $K$. The nomenclature is
justified by Theorem~5.19 of~\cite{Kelly1982Basic}, which, among other
things, says that the fully faithful $K$ has a density presentation
just when it is dense.

We will make use of density presentations in the presheaf
context~\eqref{eq:28} with respect not to the dense $K$, but to the
dense $I$. By Lemma~\ref{lem:5} we have $N_I \cong K$, and so the
$I$-absolute colimits are in this case those preserved by
$K \colon \A \rightarrow \E$. We will see numerous instances of this
situation in Section~\ref{sec:monad-theory-equiv-1} below; we give a
couple of examples now to clarify the ideas.

\begin{Exs}
  \label{ex:8}\hfill
  \begin{enumerate}[(i)]
  \item \label{item:5} Example~\ref{ex:2}\ref{item:3} corresponds to
    the presheaf context
    \begin{equation*}
      \xymatrix{
        1 \ar[r]^-{I} & \mathbb F \ar[r]^-{K} & \Set\rlap{ ,}
      }
    \end{equation*}
    and here $I$ has a density presentation given by all \emph{finite
      copowers of ${1 \in \mathbb F}$}; these are $I$-absolute since $K$
    preserves them. In fact, $\mathbb F$ has all finite coproducts and
    these are preserved by $K$, so that there is a larger density
    presentation given by \emph{all finite coproducts}~in~$\mathbb{F}$.
  \item \label{item:7} Example~\ref{ex:2}\ref{item:1} yields the
    presheaf context below, wherein $I$ has a density presentation given
    by the wide pushouts
    $[n] \cong [1]+_{[0]}[1]+_{[0]}\ldots +_{[0]} [1]$:
    \begin{equation*}
      \xymatrix{
        \mathbb G_{1} \ar[r]^-{I} &
        \Delta_{0} \ar[r]^-{K} &
        [{\mathbb G_{1}}^\mathrm{op},\Set]\rlap{ .}
      }
    \end{equation*}
  \end{enumerate}
\end{Exs}

The reason we care about density presentations is the following
result, which comprises various parts of Theorem~5.29 of
\cite{Kelly1982Basic}.

\begin{Prop}\label{prop:7}
  Let $K \colon \C \to \D$ be fully faithful and dense. The following
  are equivalent:
  \begin{enumerate}[(i),itemsep=0em]
  \item $F \colon \D \to \E$ is the pointwise left Kan extension of
    its restriction along $K$;
  \item $F$ sends $\Phi$-colimits to colimits for any density
    presentation $\Phi$ of $K$;
  \item $F$ sends $K$-absolute colimits to colimits.
  \end{enumerate}
\end{Prop}

Combined with Proposition~\ref{prop:6}, this yields the desired
sharper characterisation of the $\A$-theories and their models.

\begin{Thm}\label{thm:8}
  Let $J \colon \A \rightarrow \T$ be an $\A$-pretheory in the
  presheaf context~\eqref{eq:28}, and let $\Phi$ be a density
  presentation for $I$.
  \begin{enumerate}[(i),itemsep=0em]
  \item A functor $F \colon \T^\mathrm{op} \rightarrow \V$ is a
    $\T$-model just when
    $FJ^\mathrm{op} \colon \A^\mathrm{op} \rightarrow \V$ sends
    $\Phi$-colimits in $\A$ to limits in $\V$;
  \item $J \colon \A \rightarrow \T$ is an $\A$-theory just when it
    sends $\Phi$-colimits to colimits.
  \end{enumerate}
\end{Thm}

\subsection{Nervous monads, signatures and saturated classes}
\label{sec:nerv-monads-sign}
We now turn from characterisations for $\A$-theories to
characterisations for $\A$-nervous monads. We know from
Corollary~\ref{cor:2} that a monad is $\A$-nervous just when it is
isomorphic to $\Psi\T$ for some $\A$-pretheory
$J \colon \A \rightarrow \T$, and the examples in
Section~\ref{sec:preth-as-pres} make it an intuitively reasonable idea
that these are the monads which can be ``presented by operations and
equations with arities from $\A$''.

Our first characterisation result makes this idea precise by
exhibiting the category of $\A$-nervous monads as monadic over a
category of \emph{signatures}. We defer the proof of this result to
Section~\ref{sec:deferred-proofs}.

\begin{Defn}
  \label{def:3}
  The category $\sig$ of \emph{signatures} is the category
  $\V\text-\cat{CAT}(\ob \A, \E)$. We write
  $V \colon \mnd \rightarrow \sig$ for the functor sending
  $\mathsf{T}$ to $(Ta)_{a \in \A}$.
\end{Defn} 

\begin{restatable}{Thm}{thmseven}
  \label{thm:7}
  $V \colon \mnd \rightarrow \sig$ has a left adjoint
  $F \colon \sig \rightarrow \mnd$ taking values in $\A$-nervous
  monads. Moreover:
  \begin{enumerate}[(i),itemsep=0em]
  \item The restricted functor $V \colon \nerv \rightarrow \sig$ is
    monadic;
  \item A monad $\mathsf{T} \in \mnd$ is $\A$-nervous if and only if
    it is a colimit in $\mnd$ of monads in the image of $F$;
  \item Each of $\nerv$, $\preth$ and $\th$ is locally presentable.
  \end{enumerate}
\end{restatable}

The idea behind this result originates in~\cite{Kelly1993Adjunctions}.
A signature $\Sigma \in \sig$ specifies for each $a \in \A$ an
$\E$-object $\Sigma a$ of ``operations of input arity $a$''. The free
monad $F\Sigma$ on this signature has as its algebras the
\emph{$\Sigma$-structures}: objects $X \in \E$ endowed with a function
$\E(a,X) \rightarrow \E(\Sigma a,X)$ for each $a \in \A$. The above
result implies that a monad $\mathsf{T} \in \mnd$ is $\A$-nervous just
when it admits a presentation as a coequaliser
$F\Gamma \rightrightarrows F\Sigma \twoheadrightarrow
\mathsf{T}$---that is, a presentation by a signature $\Sigma$ of basic
operations together with a family $\Gamma$ of equations between
derived operations.

We now turn to our second characterisation result for $\A$-nervous
monads. This is motivated by the fact, noted in the introduction, that
in many monad--theory correspondences the class of monads can be
characterised by a colimit-preservation property. To reproduce this
result in our setting, we require a closure property of the arities in
the subcategory $\A$ which, roughly speaking, says that substituting
$\A$-ary operations into $\A$-ary operations again yields $\A$-ary
operations.

\begin{Defn}
  \label{def:4}
  An endo-$\V$-functor $F \colon \E \rightarrow \E$ is called
  \emph{$\A$-\induced} if it is the pointwise left Kan extension of
  its restriction along $K$. We call $\A$ a \emph{saturated class of
    arities} if $\A$-\induced endofunctors of $\E$ are closed under
  composition.
\end{Defn}

\begin{Ex}
  \label{ex:9}
  In the case of $K \colon \mathbb{F} \hookrightarrow \cat{Set}$,
  there is a density presentation for $K$ given by \emph{all filtered
    colimits in $\cat{Set}$}, so that by Proposition~\ref{prop:7}, an
  endofunctor $\cat{Set} \rightarrow \cat{Set}$ is
  $\mathbb{F}$-\induced just when it preserves filtered colimits. Thus
  $\mathbb{F} \hookrightarrow \cat{Set}$ is a saturated class of
  arities.
\end{Ex}

\begin{Ex}
  \label{ex:11}
  More generally, if $\Phi$ is a class of enriched colimit-types and
  ${K \colon \A \rightarrow \E}$ exhibits $\E$ as the free
  cocompletion of $\A$ under $\Phi$-colimits, then there is a density
  presentation of $K$ given by all $\Phi$-colimits, and an endofunctor
  of $\E$ is $K$-\induced just when it preserves $\Phi$-colimits. Thus
  $\A$ is a saturated class of arities.
\end{Ex}

\begin{Ex}
  \label{ex:10}
  Let $K \colon \A \hookrightarrow \Set$ be the inclusion of the
  one-object full subcategory $\A$ on the two-element set
  $2 = \{0,1\}$. Since the dense generator $1$ of $\cat{Set}$ is a
  retract of $2$, and taking retracts does not change categories of
  presheaves, $\A$ is dense in $\cat{Set}$. We claim it does
  \emph{not} give a saturated class of arities.

  To see this, note first that
  $(\thg)^2 \colon \cat{Set} \rightarrow \cat{Set}$ is $\A$-\induced,
  being a left Kan extension along $K$ of the representable
  $\A(2, \thg) \colon \A \rightarrow \cat{Set}$. We claim that
  $(\thg)^2 \circ (\thg)^2$ is \emph{not} $\A$-\induced. For indeed,
  by the Yoneda lemma, any $X \in [\A, \cat{Set}]$ has an epimorphic
  cover by copies of the unique representable $\A(2, \thg)$. Since
  left Kan extension preserves epimorphisms, each $\mathrm{Lan}_K(X)$
  admits an epimorphic cover by copies of $(\thg)^2$. But
  $(\thg)^2 \circ (\thg)^2 \cong (\thg)^4$ can admit no such cover,
  since the identity map on $4$ does not factor through $2$, and so
  cannot be $\A$-\induced.
\end{Ex}

The proof of the following result will again be deferred to
Section~\ref{sec:deferred-proofs} below.

\begin{restatable}{Thm}{thmfive}
  \label{thm:5}
  Let $\A$ be a saturated class of arities in $\E$. The following are
  equivalent properties of a monad $\mathsf{T} \in \mnd$:
  \begin{enumerate}[(i),itemsep=0em]
  \item $\mathsf{T}$ is $\A$-nervous;
  \item $T \colon \E \rightarrow \E$ is $\A$-\induced;
  \item $T \colon \E \rightarrow \E$ preserves $\Phi$-colimits for any
    density presentation $\Phi$ of $K$.
  \end{enumerate}
\end{restatable}

\subsection{The monad--theory equivalence in practice}
\label{sec:monad-theory-equiv-1}

We now apply our characterisation results to the examples of
Section~\ref{sec:setting}. In many cases, the explicit descriptions we
obtain of the $\A$-nervous monads, the $\A$-theories, and their models
will allow us to reconstruct a familiar monad--theory correspondence
from the literature.

\begin{Exs}\label{exs:TheoriesDensity}
  As before, we begin with the unenriched examples where
  $\V = \cat{Set}$.

  \begin{enumerate}[(i)]
  \item The case $\E = \cat{Set}$ and $\A = \mathbb{F}$ corresponds to
    the instance of the presheaf context described in
    Examples~\ref{ex:8}\ref{item:5}. Applied to the density
    presentations for $I$ given there, Theorem~\ref{thm:8} tells us
    that an $\mathbb{F}$-pretheory
    $J \colon \mathbb{F} \rightarrow \T$ is an $\mathbb{F}$-theory
    just when it preserves finite copowers of $1$, or equally (using
    the larger density presentation) all finite coproducts. It thus
    follows that the $\mathbb{F}$-theories are the \emph{Lawvere
      theories} of~\cite{Lawvere1963Functorial}. Moreover a functor
    $F \colon \T^\mathrm{op} \to \Set$ is a $\T$-model if and only if
    $FJ^\mathrm{op} \colon \mathbb F^\mathrm{op} \to \Set$ preserves
    finite products. Since, in this case, $J$ also \emph{reflects}
    finite coproducts, this happens just when
    $F \colon \T^\mathrm{op} \to \Set$ is itself
    finite-product-preserving, that is, just when $F$ is a model of
    the Lawvere theory $\T$.

    On the other hand, by Example~\ref{ex:9}, $\mathbb{F}$ is a
    saturated class of arities, and the $\mathbb{F}$-\induced
    endofunctors are the finitary ones; so by Theorem~\ref{thm:5}, a
    monad on $\cat{Set}$ is $\mathbb{F}$-nervous just when it is
    finitary. Theorem~\ref{thm:1} thus specialises to the classical
    \emph{finitary monad--Lawvere theory} correspondence, while
    Theorem~\ref{thm:9} recaptures its compatibility with
    semantics.\label{item:8}

  \item When $\E$ is locally finitely presentable and $\A = \E_{f}$,
    the category of $K$-nerves is,
    by~\cite[Kollar~7.9]{Gabriel1971Lokal}, the full subcategory of
    $[\E_f^\mathrm{op}, \cat{Set}]$ on the
    \emph{finite-limit-preserving} functors. So an $\E_f$-pretheory
    $J \colon \E_f \rightarrow \T$ is an $\E_f$-theory just when each
    $\T(J\thg, a) \colon {\E_f}^\mathrm{op} \rightarrow \cat{Set}$
    preserves finite limits. By the Yoneda lemma, this happens just
    when $J$ preserves finite colimits, so that the $\E_f$-theories
    are precisely~\cite{Nishizawa2009Lawvere}'s \emph{Lawvere
      $\E$-theories}.

    The concrete $\T$-models in this setting are exactly the models
    of~\cite[Definition~2.2]{Nishizawa2009Lawvere}. The general
    $\T$-models are those functors
    $F \colon \T^\mathrm{op} \rightarrow \cat{Set}$ for which
    $FJ^\mathrm{op} \colon \smash{\E_f^\mathrm{op}} \rightarrow
    \cat{Set}$ is a $K$-nerve, i.e., finite-limit-preserving; these
    are the more general models
    of~\cite[Definition~12]{Lack2009Gabriel-Ulmer}, and the
    correspondence between the two notions in Proposition~\ref{prop:1}
    recaptures Proposition~15 of~\emph{ibid}.

    On the monad side, since $K \colon \E_f \rightarrow \E$ exhibits
    $\E$ as the free filtered-colimit completion of $\E_f$,
    Example~\ref{ex:11} and Theorem~\ref{thm:5} imply that $\E_f$ is a
    saturated class, and that the $\E_f$-nervous monads are the
    finitary ones. So in this case, Theorem~\ref{thm:1} and
    Corollary~\ref{cor:5} reconstruct (the unenriched version of) the
    monad--theory correspondence given
    in~\cite[Theorem~5.2]{Nishizawa2009Lawvere}.\label{item:9}

  \item More generally, when $\E$ is locally $\lambda$-presentable and
    $\A = \E_\lambda$ is a skeleton of the full subcategory of
    $\lambda$-presentable objects, the $\E_\lambda$-theories are those
    pretheories $J \colon \E_\lambda \rightarrow \T$ which preserve
    $\lambda$-small colimits; the $\T$-models are functors
    $F \colon \T^\mathrm{op} \rightarrow \cat{Set}$ for which
    $FJ^\mathrm{op}$ preserves $\lambda$-small limits; and the
    $\E_\lambda$-nervous monads are those whose endofunctor preserves
    $\lambda$-filtered colimits.

  \item When $\E = [\mathbb{G}_1^\mathrm{op}, \cat{Set}]$ and
    $\A = \Delta_0$, we are in the presheaf context of
    Examples~\ref{ex:8}\ref{item:7}. For the density presentation for
    $I$ given there, Theorem~\ref{thm:8} tells us that a pretheory
    $J \colon \Delta_0 \rightarrow \T$ is a $\Delta_0$-theory just
    when it preserves the wide pushouts
    $[n] \cong [1]+_{[0]}[1]+_{[0]}\ldots +_{[0]} [1]$. Moreover, a
    functor $ X \colon \T^\mathrm{op} \rightarrow \cat{Set}$ is a
    $\T$-model just when it sends each of these wide pushouts to a
    limit in $\cat{Set}$. This is precisely the \emph{Segal condition}
    of \cite{Segal1968Classifying}; in elementary terms, it requires
    the invertibility of each canonical map
    \begin{equation}\label{eq:37}
      Xn \longrightarrow X1 \times_{X0} X1 \times_{X0} \cdots \times_{X0} X1\rlap{ .}
    \end{equation}

    In Corollary~\ref{cor:notsaturated} below we will see that $\Delta_0$
    is \emph{not} a saturated class of arities, and so we have no more
    direct characterisations of the $\Delta_0$-nervous monads than is
    given by Corollary~\ref{cor:2} or Theorem~\ref{thm:7}. However,
    Example~\ref{ex:3} provides us with natural examples of
    $\Delta_0$-nervous monads: namely, the monads $\mathsf{T}$ and
    $\mathsf{T}_g$ for \emph{categories} and for \emph{groupoids} on
    $[\mathbb{G}_1^\mathrm{op}, \cat{Set}]$. As was already noted
    in~\cite{Weber2007Familial}, the nervosity of~$\mathsf{T}$ recaptures
    the classical nerve theorem relating categories and simplicial sets.
    Indeed, the $\Delta_0$-theory associated to $\mathsf{T}$ is the first
    part of the (bijective-on-objects, fully faithful) factorisation
    \begin{equation*}
      \Delta_0 \xrightarrow{J_\mathsf{T}} \Delta
      \xrightarrow{K_\mathsf{T}} \cat{Cat}
    \end{equation*}
    of the composite
    $F_\mathsf{T}K \colon \Delta_0 \rightarrow \cat{Cat}$. The
    interposing object here is the topologist's simplex category
    $\Delta$, with $K_T$ the standard inclusion into $\cat{Cat}$. Thus,
    to say that $\mathsf{T}$ is $\Delta_0$-nervous is to say that:
    \begin{enumerate}[itemsep=0em]
    \item The classical nerve functor
      $N_{K_\mathsf{T}} \colon \cat{Cat} \rightarrow [\Delta^\mathrm{op},
      \cat{Set}]$ is fully faithful;
    \item The essential image of $N_{K_\mathsf{T}}$ comprises those
      $X \in [\Delta^\mathrm{op}, \cat{Set}]$ for which $XJ_\mathsf{T}$
      is a $K$-nerve.
    \end{enumerate}
    This much is already done in~\cite{Weber2007Familial}, but our use of
    density presentations allows for a small improvement. To say that
    $XJ_\mathsf{T}$ is a $K$-nerve in (b) is equally to say that $X$ is a
    $\T$-model, or equally that $X$ satisfies the Segal condition
    expressed by the invertibility of each~\eqref{eq:37}. This is a mild
    sharpening of~\cite{Weber2007Familial}, where the ``Segal condition''
    is left in the abstract form given in (b) above.

    In a similar way, the nervosity of the monad $\mathsf T_g$ for small
    groupoids captures the ``symmetric nerve theorem''. This states that
    the functor $\Gpd \rightarrow [\mathbb{F}_+^\mathrm{op}, \cat{Set}]$
    sending a groupoid to its \emph{symmetric nerve}---indexed by the
    category of non-empty finite sets---is fully faithful, and
    characterises the essential image once again as the functors
    satisfying the Segal condition~\eqref{eq:37}.

  \item With $\E = [\mathbb{G}^\mathrm{op}, \cat{Set}]$ and
    $\A = \Theta_0$, we are now in the presheaf context
    \begin{equation*}
      \xymatrix{
        \mathbb G \ar[r]^-{I} & \Theta_{0} \ar[r]^-{K} & [\mathbb G^\mathrm{op},\Set]\rlap{ .}
      }
    \end{equation*}
    $I$ has a density presentation given by the \emph{$I$-globular sums}
    $(n_{1},\ldots,n_{k}) \cong (n_{1}) +_{(n_{2})} + (n_{3}) + \ldots
    +_{(n_{k-1})} (n_{k})$ in $\Theta_0$; whence by Theorem~\ref{thm:8},
    a pretheory $J \colon \Theta_0 \rightarrow \T$ is a $\Theta_0$-theory
    when it preserves these $I$-globular sums---that is, when it is a
    \emph{globular theory} in the sense of
    \cite{Berger2002A-Cellular}\begin{footnote}{The definition of
        globular theory in~\cite{Berger2002A-Cellular} has the extra
        condition, satisfied in most cases, that $J$ be a faithful
        functor.}\end{footnote}. A functor
    $F \colon \T^\mathrm{op} \rightarrow \cat{Set}$ is a $\T$-model when
    it sends $I$-globular sums to limits, thus when each map
    \begin{equation*}
      X\vec n \longrightarrow Xn_{1} \times_{Xn_{2}} Xn_{3} \times _{Xn_{4}} \ldots \times_{Xn_{k-1}} Xn_{k}
    \end{equation*}
    is invertible. Once again, $\Theta_0$ is \emph{not} a saturated class
    of arities, and so there is no direct characterisation of the
    $\Theta_0$-nervous monads; however, their interaction with
    $\Theta_0$-theories is important in the literature on globular
    approaches to higher category theory, as we now outline.

    Globular theories can describe structures on globular sets such as
    strict or weak $\omega$-categories and $\omega$-groupoids. For the
    strict variants, we pointed out in Section~\ref{sec:preth-as-pres-1}
    that these may be modelled by $\Theta_0$-pretheories; and since
    reflecting a pretheory $\T$ into a theory $\Phi \Psi \T$ does not
    change the models, it is immediate that there are $\Theta_0$-theories
    modelling these structures too.

    The original definition of globular weak $\omega$-category was given
    by Batanin in~\cite{Batanin1998Monoidal}; he defines them be globular
    sets equipped with algebraic structure controlled by a \emph{globular
      operad}. Globular operads can be understood as certain cartesian
    monads on globular sets. Berger~\cite{Berger2002A-Cellular}
    introduced globular theories and described the passage from a
    globular operad $T$ to a globular theory $\Theta_T$ just as in
    Section~\ref{sec:from-monads-preth} above. In our language, his
    Theorem~1.17 states exactly that each globular operad $T$ is
    $\Theta_0$-nervous, so that algebras for the globular operad are the
    same as models of the associated theory $\Theta_T$. In particular,
    Batanin's weak $\omega$-categories are the models of a globular
    theory\begin{footnote}{As an aside, we note that a complete
        understanding of those globular theories corresponding to
        globular operads was obtained in Theorem 6.6.8 of
        \cite{Ara2010Sur-les-infty-groupoides}. See also Section 3.12 of
        \cite{Berger2012Monads}.}\end{footnote}. On the other hand,
    Grothendieck weak
    $\omega$-groupoids~\cite{Maltsiniotis2010Grothendieck} are, \emph{by
      definition}, models for certain globular theories called
    \emph{coherators}.\label{item:14}
  \end{enumerate}
  We now proceed to our examples over a more general base for enrichment
  $\V$.
  \begin{enumerate}[(i)]
    \addtocounter{enumi}{5}
  \item With $\V = \E$ a locally finitely presentable symmetric monoidal
    category and with $\A = \V_f$, we are in the presheaf context
    \begin{equation*}
      \xymatrix{
        \I \ar[r]^-{I} & \V_{f} \ar[r]^-{K} & \V\rlap{ ,}
      }
    \end{equation*}
    wherein $I$ has a density presentation given by the class of all
    \emph{finite tensors}--- tensors by finitely presentable objects of
    $\V$. Thus by Theorem~\ref{thm:8}, the $\V_f$-theories are the
    pretheories $J \colon \V_f \rightarrow \T$ which preserve finite
    tensors, which are precisely the \emph{Lawvere $\V$-theories}
    of~\cite[Definition~3.1]{Power1999Enriched}. Furthermore, like
    in~\ref{item:8}, a functor $F \colon \T^\mathrm{op} \to \V$ is a
    $\T$-model just when it preserves finite cotensors, just as in
    Definition~3.2 of \emph{ibid}. On the other hand,
    $\V_f \rightarrow \V$ exhibits $\V$ as the free filtered-colimit
    completion of $\V_f$; whence by Example~\ref{ex:11} it is a saturated
    class of arities, and by Theorem~\ref{thm:5} the $\V_f$-nervous
    monads are again the finitary ones. So Theorems~\ref{thm:1}
    and~\ref{thm:9} specialise to Theorems~4.3, 3.4 and 4.2
    of~\cite{Power1999Enriched}.

  \item Now taking $\E$ to be any locally finitely presentable
    presentable $\V$-category and $\A = \E_f$, we may argue as
    in~\ref{item:9} to recapture the fully general enriched
    monad--theory correspondence of~\cite{Nishizawa2009Lawvere}, and its
    interaction with semantics.

  \item Now suppose we are in the situation of
    Examples~\ref{ex:2}\ref{item:4}, provided with a class $\Phi$ of
    enriched colimit-types satisfying Axiom~A of~\cite{Lack2011Notions}.
    With $\E = \V$ and $\A = \V_\Phi$, we are now in the presheaf
    context
    \begin{equation*}
      \xymatrix{
        \I \ar[r]^-{I} & \V_{\Phi} \ar[r]^-{K} & \V\rlap{ .}
      }
    \end{equation*}
    By~\cite[Theorem~5.35]{Kelly1982Basic}, $I$ has a density
    presentation given by $\Phi$-tensors (i.e., tensors by objects in
    $\Phi$) while by~\cite[Theorem~3.1]{Lack2011Notions}, $K$ exhibits
    $\V$ as the free $\Phi$-flat cocompletion of $\V_\Phi$. Arguing as in
    the preceding parts, we see that $\V_\Phi$-theories are pretheories
    $J \colon \V_\Phi^\mathrm{op} \rightarrow \T$ which preserve
    $\Phi$-tensors, that $\T$-models are $\Phi$-tensor-preserving
    functors $F \colon \T^\mathrm{op} \rightarrow \V$, and that a monad
    is $\V_\Phi$-nervous if its underlying endofunctor preserves
    $\Phi$-flat colimits. This sharpens slightly the results obtained
    in~\cite{Lack2011Notions} in the special case $\E = \V$.

  \item Finally, in the situation of Examples~\ref{ex:2}\ref{item:10},
    we find that the $\A$-theories are the $\Phi$-colimit preserving
    pretheories $J \colon \A \rightarrow \T$; that the $\T$-models are
    functors $F \colon \T^\mathrm{op} \rightarrow \V$ such that
    $FJ^\mathrm{op}$ preserves $\Phi$-limits; and that a monad is
    $\A$-nervous just when it preserves $\Phi$-flat colimits. In this
    way, our Theorems~\ref{thm:1} and~\ref{thm:9} reconstruct
    Theorems~7.6 and 7.7 of~\cite{Lack2011Notions}.
  \end{enumerate}
\end{Exs}

\section{Monads with arities and theories with arities}
\label{sec:arities}

In the introduction, we mentioned the general framework for
monad--theory correspondences obtained in~\cite{Weber2007Familial,
  Berger2012Monads}. Similar to this paper, the basic setting involves
a category $\E$ and a small, dense subcategory
$K \colon \A \hookrightarrow \E$; given these data, one defines
notions of \emph{monad with arities $\A$} and \emph{theory with
  arities $\A$}, and proves an equivalence between the two that is
compatible with semantics.

In this section, we compare this framework with ours by comparing the
classes of monads and of theories. We will see that our setting yields
\emph{strictly} larger classes of monads and theories which are
better-behaved in practically useful ways. On the other hand, in the
more restrictive setting of~\cite{Weber2007Familial,
  Berger2012Monads}, checking that a monad or theory is in the
required class may give greater combinatorial insight into the
structure which it describes.

\subsection{Monads with arities versus nervous monads}
\label{sec:monads-with-arities-1}

In~\cite{Weber2007Familial, Berger2012Monads} the authors work in the
\emph{unenriched} setting; the introduction to~\cite{Berger2012Monads}
states that the results ``should be applicable'' also in the enriched
one. To ease the comparison to our results, we take it for granted
that this is true, and transcribe their framework into the enriched
context without further comment.

Another difference is that we assume local presentability of $\E$
while~\cite{Weber2007Familial} assumes only \emph{cocompleteness},
and~\cite{Berger2012Monads} not even that. Given a small dense
subcategory, there is no readily discernible difference between
cocompleteness and local presentability\footnote{Indeed, if there
  were, then it would negate the large cardinal axiom known as
  \emph{Vop\v enka's principle}~\cite[Chapter 6]{Adamek1994Locally}.};
however, cocompleteness is substantively different from nothing, so
that in this respect~\cite{Berger2012Monads}'s results are more
general than ours. However all known applications are in the context
of a locally presentable $\E$, and so we do not lose much in
restricting to this context. In conclusion, when we make our
comparison we will work in exactly the same general setting as in
Section~\ref{sec:setting}, and now have:
\begin{Defn}
  \label{def:5}
  \cite[Definition~4.1]{Weber2007Familial} An endofunctor
  $T \colon \E \rightarrow \E$ is said to \emph{have arities $\A$} if
  the composite $\V$-functor
  $N_KT \colon \E \rightarrow [\A^\mathrm{op}, \V]$ is the left Kan
  extension of its own restriction along $K$. A monad
  $\mathsf{T} \in \mnd$ is a \emph{monad with arities $\A$} if its
  underlying endofunctor has arities $\A$.
\end{Defn}

We consider the following way of restating this to be illuminating.

\begin{Prop}\label{prop:10}
  An endofunctor $T \colon \E \rightarrow \E$ has arities $\A$ if and
  only if it sends $K$-absolute colimits to $K$-absolute colimits. In
  particular, each endofunctor with arities $\A$ is $\A$-\induced.
\end{Prop}
\begin{proof}
  By Proposition~\ref{prop:7}, $T$ has arities $\A$ just when $N_KT$
  sends $K$-absolute colimits to colimits. Since $N_K$ is fully
  faithful, it reflects colimits, and so $\mathsf{T}$ has arities $\A$
  just when $T$ sends $K$-absolute colimits to colimits which are
  preserved by $N_K$---that is, to $K$-absolute colimits.

  For the second claim, recall from Definition~\ref{def:4} that an
  endofunctor $T \colon \E \rightarrow \E$ is $\A$-\induced if it is
  the left Kan extension of its own restriction to $\A$, or
  equivalently, by Proposition~\ref{prop:7}, when it sends
  $K$-absolute colimits to colimits.
\end{proof}

Recall also that we call a class of arities $\A$ \emph{saturated} when
$\A$-\induced endofunctors are closed under composition.
Example~\ref{ex:10} shows that this condition is not always satisfied.
In light of the preceding result, the endofunctors with arities $\A$
can be seen as a natural subclass of the $\A$-\induced endofunctors
for which composition-closure is \emph{always} verified.

The reason that Weber introduced monads with arities was in order to
prove his \emph{nerve theorem}~\cite[Theorem 4.10]{Weber2007Familial},
which in our language may be restated as:

\begin{Thm}\label{thm:11}
  Monads with arities $\A$ are $\A$-nervous.
\end{Thm}

One may reasonably ask whether the classes of monads with arities and
$\A$-nervous monads in fact coincide. In many cases, this is true; in
particular, in the situation of Example~\ref{ex:11}, where
$K \colon \A \rightarrow \E$ exhibits $\E$ as the free
$\Phi$-cocompletion of $\A$ for some class of colimit-types $\Phi$.
Indeed, this condition implies that a monad $\mathsf{T}$ is
$\A$-nervous precisely when $T$ sends $\Phi$-colimits to
$\Phi$-colimits; since $\Phi$-colimits are $K$-absolute, this in turn
implies that $N_KT$ sends $\Phi$-colimits to colimits, and so is the
left Kan extension of its own restriction along $K$. So in this case,
every $\A$-nervous monad has arities $\A$; so in particular, the two
notions coincide in each of Examples~\ref{ex:2}\ref{item:3},
\ref{item:11}, \ref{item:12}, \ref{item:6}, \ref{item:13},
\ref{item:4} and \ref{item:10}.

However, they do \emph{not} coincide in general. That is, in some
instances of our basic setting, there exist monads which are
$\A$-nervous but do not have arities $\A$. We give three examples of
this. The first two arise in the setting of
Example~\ref{ex:2}\ref{item:1}, and concern the monads for groupoids
and involutive graphs respectively.

\begin{Prop}\label{prop:groupoids}
  The monad $\mathsf{T}$ on
  $\Gph \defeq [\mathbb G_{1}^\mathrm{op},\Set]$ whose algebras are
  groupoids is $\Delta_0$-nervous but does not have
  $\Delta_0$-\induced underlying endofunctor. It follows that
  $\mathsf{T}$ does not have arities $\Delta_0$.
\end{Prop}
\begin{proof}
  From Example~\ref{ex:3} we know that $\mathsf T$ is
  $\Delta_0$-nervous. To see that $T$ is not $\Delta_0$-\induced,
  consider the graph $X$ with vertices and arrows as to the left in:
  \begin{equation}\label{eq:pushout}
    a \xrightarrow{r} b \xleftarrow{s} c \qquad \qquad
    \cd[@-1em]{
      {[0]} \ar[r]^-{\tau} \ar[d]_{\tau} &
      {[1]} \ar[d]^{s} \\
      {[1]} \ar[r]^-{r} &
      {X}\rlap{ .}
    }
  \end{equation}
  This $X$ is equally the $K$-absolute pushout right above; so if $T$
  were $\Delta_0$-\induced then it would preserve this pushout. But
  $T[1]+_{T[0]} T[1]$ is the graph
  \begin{equation*}
    \cd{
      a  \ar@(ul,dl)_{1_a} \ar@<-3pt>[r]_-{r} & b \ar@<-3pt>[l]_-{r^{-1}}
      \ar@{<-}@<-3pt>[r]_-{s} \ar@(ul,ur)^{1_b} & c \ar@{<-}@<-3pt>[l]_-{s^{-1}}
      \ar@(ur,dr)^{1_c}
    }
  \end{equation*}
  wherein, in particular, there is no edge $a \to c$; while in $TX$ we
  have $s^{-1} \circ r \colon a \rightarrow c$. So the pushout is not
  preserved. This shows that $T$ is not $\Delta_0$-\induced and so, by
  Proposition~\ref{prop:10}, that $\mathsf{T}$ does not have arities
  $\Delta_0$.
\end{proof}

Since the above result exhibits a $\Delta_0$-nervous monad whose
underlying endofunctor is not $\Delta_0$-\induced, we can apply
Theorem~\ref{thm:5} to deduce:

\begin{Cor}\label{cor:notsaturated}
  $K \colon \Delta_0 \hookrightarrow \Gph$ is not a saturated class of
  arities.
\end{Cor}

Our second example, originally due to
Melli\`es~\cite[Appendix~III]{Mellies2010Segal}, shows that even
monads with $\Delta_0$-\induced endofunctor need not have arities
$\Delta_0$. In this example, we call a graph
$s, t \colon X_{1} \rightrightarrows X_{0}$ \emph{involutive} if it
comes endowed with an order-$2$ automorphism
$i \colon X_1 \rightarrow X_1$ reversing source and target, i.e., with
$si = t$ (and hence also $ti = s$).

\begin{Prop}\label{prop:12}
  The monad $\mathsf{T}$ on
  $\Gph \defeq [\mathbb G_{1}^\mathrm{op},\Set]$ whose algebras are
  involutive graphs is $\Delta_0$-nervous and has $\Delta_0$-\induced
  underlying endofunctor, but does not have arities $\Delta_0$.
\end{Prop}
\begin{proof}
  The value of $T$ at $s,t \colon X_1 \rightrightarrows X_0$ is given
  by $\spn{s,t}, \spn{t,s} \colon X_1 + X_1 \rightrightarrows X_0$. It
  follows that $T$ is \emph{cocontinuous} and so certainly
  $\Delta_0$-\induced. To see it does not have arities $\Delta_0$,
  consider again the graph~\eqref{eq:pushout} and its $K$-absolute
  pushout presentation. If this were preserved by
  $N_KT \colon \Gph \rightarrow [\Delta_0^\mathrm{op}, \cat{Set}]$
  then, on evaluating at $[2]$, the maps
  $\Gph([2], T[1]) \rightrightarrows \Gph([2], TX)$ given by
  postcomposition with $Tr$ and $Ts$ would be jointly surjective. To
  show this is not so, consider the map $f \colon [2] \rightarrow TX$
  picking out the composable pair
  $(r \colon a \rightarrow b, i(s) \colon b \rightarrow c)$. Since
  neither $Tr$ nor $Ts$ are surjective on objects, the
  bijective-on-objects $f$ cannot factor through either of them. This
  shows that $\mathsf{T}$ does not have arities $\Delta_0$.
\end{proof}

Our final example shows that not even free monads on
$\A$-signatures---which are $\A$-nervous by Theorem~\ref{thm:7}
above---need necessarily have arities $\A$.

\begin{Prop}
  \label{prop:9}
  Let $\V = \E = \cat{Set}$ and let $\A$ be the one-object full
  subcategory on a two-element set. The free monad on the terminal
  $\A$-signature does not have $\A$-\induced underlying endofunctor
  and therefore does not have arities $\A$.
\end{Prop}
\begin{proof}
  The algebras for the free monad $\mathsf{T}$ on the terminal
  signature are sets equipped with a binary operation. Elements of the
  free $\mathsf{T}$-algebra on $X$ are binary trees with leaves
  labelled by elements of $X$, yielding the formula
  \begin{equation*}
    TX = \textstyle\sum_{n \in \mathbb{N}} C_n \times X^{n+1}
  \end{equation*}
  where $C_n$ is the $n$th Catalan number. In particular, $T$ contains
  at least one coproduct summand $(\thg)^4$ and so, as in
  Example~\ref{ex:10}, is not $\A$-\induced; in particular, by
  Proposition~\ref{prop:10}, it does not have arities $\A$.
\end{proof}

\subsection{Theories with arities $\A$ versus $\A$-theories}
\label{sec:theor-with-arit}

The paper \cite{Berger2012Monads} introduced \emph{theories with
  arities} $\A$. These are $\A$-pretheories $J \colon \A \to \T$ for
which the composite
\begin{equation}\label{eq:theorywitharities}
  [\A^\mathrm{op},\V] \xrightarrow{\mathrm{Lan}_J} [\T^\mathrm{op},\V]
  \xrightarrow{[J^\mathrm{op}, 1]} [\A^\mathrm{op},\V]
\end{equation}
takes $K$-nerves to $K$-nerves. This functor takes the representable
$\A(\thg, x)$ to $\T(J\thg, x)$, so that in this language, we may
describe the $\A$-theories as the pretheories for
which~\eqref{eq:theorywitharities} takes each representable to a
$K$-nerve. It follows that:

\begin{Prop}
  Theories with arities $\A$ are $\A$-theories.
\end{Prop}

\begin{proof}
  It suffices to observe that each representable $\A(\thg, x)$ is a
  $K$-nerve since $\A(\thg,x) \cong \E(K\thg,Kx) = N_{K}(Kx)$.
\end{proof}

Theorem 3.4 of~\cite{Berger2012Monads} establishes an equivalence
between the categories of monads with arities $\A$ and of theories
with arities $\A$. The functor taking a monad with arities to the
corresponding theory with arities is defined in the same way as
the~$\Phi$ of Section~\ref{sec:from-monads-preth}, and so it follows
that:

\begin{Prop}
  \label{prop:11}
  The equivalence of monads with arities $\A$ and theories with
  arities $\A$ is a restriction of the equivalence between
  $\A$-nervous monads and $\A$-theories.
\end{Prop}

In particular, there exist $\A$-theories which are not theories with
arities $\A$; it is this statement which was verified in
in~\cite[Appendix~III]{Mellies2010Segal}.

\subsection{Colimits of monads with arities}
\label{sec:colimits-monads-with}

In Theorem~\ref{thm:7} we saw that the $\A$-nervous monads are the
closure of the free monads on $\A$-signatures under colimits in
$\mnd$. Since colimits of monads are algebraic, this allows us to give
intuitive presentations for $\A$-nervous monads as suitable colimits
of frees. The pretheory presentations of
Section~\ref{sec:preth-as-pres} can be understood as particularly
direct descriptions of such colimits.

Since not every $\A$-nervous monad has arities $\A$, the monads with
arities are \emph{not} the colimit-closure of the frees on signatures.
We already saw one explanation for this in Proposition~\ref{prop:9}:
the free monads on signatures need not have arities. However, this
leaves open the possibility that the monads with arities $\A$ are the
colimit-closure of some smaller class of basic monads---which would
allow for the same kind of intuitive presentation as we have for
$\A$-nervous monads. The following result shows that even this is not
the case.

\begin{Thm}
  \label{thm:12}
  Monads with arities $\A$ need not be closed in $\mnd$ under
  colimits.
\end{Thm}

\begin{proof}
  We saw in Proposition~\ref{prop:12} that, when $\E = \Gph$ and
  $\A = \Delta_0$, the monad $\mathsf{T}$ for involutive graphs does
  not have arities $\Delta_0$. To prove the result it will therefore
  suffice to exhibit $\mathsf{T}$ as a colimit in $\cat{Mnd}(\Gph)$ of
  a diagram of monads with arities $\Delta_0$. This diagram will be a
  coequaliser involving a pair of monads $\mathsf{P}$ and
  $\mathsf{Q}$, whose respective algebras are:
  \begin{itemize}[itemsep=0\baselineskip]
  \item For $\mathsf{P}$: graphs $X$ endowed with a function
    $u \colon X_1 \rightarrow X_0$;
  \item For $\mathsf{Q}$: graphs $X$ endowed with an order-$2$
    automorphism $i \colon X_1 \rightarrow X_1$.
  \end{itemize}
  We construct this coequaliser of monads in terms of the categories
  of algebras. The category $\Gph^\mathsf{T}$ of involutive graphs is
  an equaliser in $\cat{CAT}$ as to the left in:
  \begin{equation*}
    \cd{
      \Gph^\mathsf{T} \ar@{ >->}[r]^-{E} & \Gph^\mathsf{Q}
      \ar@<3pt>[r]^-{F} \ar@<-3pt>[r]_-{G} & \Gph^\mathsf{P} & &
      \mathsf{P} \ar@<3pt>[r]^-{\varphi} \ar@<-3pt>[r]_-{\gamma} &
      \mathsf{Q} \ar@{->>}[r]^-{\varepsilon} & \mathsf{T}
    }
  \end{equation*}
  where the functors $F$ and $G$ send a $\mathsf{Q}$-algebra $(X,i)$ to
  the respective $\mathsf{P}$-algebras $(X,si)$ and $(X,t)$. Since each
  of these functors commutes with the the forgetful functors to $\Gph$,
  we have an equaliser of forgetful functors in $\CAT/\Gph$. Since the
  functor
  $\mathrm{Alg} \colon \cat{Mnd}(\Gph)^\mathrm{op} \to \CAT/\Gph$ is
  fully faithful, this equaliser must be the image of a coequaliser
  diagram in $\cat{Mnd}(\Gph)$ as right above.

  It remains to show that in this coequaliser presentation both
  $\mathsf{P}$ and $\mathsf{Q}$ have arities $\Delta_0$. By
  Proposition~\ref{prop:7}, this means showing that $N_KP$ and $N_KQ$
  send $K$-absolute colimits to colimits, or equally, that each
  $\Gph([n], P\thg)$ and $\Gph([n], Q\thg)$ sends $K$-absolute colimits
  to colimits. To see this, we calculate $P$ and $Q$ explicitly. On the
  one hand, the free $\mathsf{P}$-algebra on a graph $X$ is obtained by
  freely adjoining an element $u(f)$ to $X_0$ for each $f \in X_1$. On
  the other hand, the free $\mathsf{Q}$-algebra on $X$ is obtained by
  freely adjoining an element $i(f) \in X_1$ for each $f \in X_1$. Thus
  we have
  \begin{equation*}
    PX = X + X_1 \cdot [0] \qquad \text{and} \qquad QX = X + X_1 \cdot [1]
  \end{equation*}
  where we use $\cdot$ to denote copower. Since each $[n] \in \Gph$ is
  connected, and since each hom-set $\cat{Gph}([n],[m])$ has
  cardinality $\max(0,m-n+1)$, we conclude that
  \begin{equation}
    \label{eq:36}
    \begin{aligned}
      \Gph([n], PX) &=
      \begin{cases}
        \Gph([0],X) + \Gph([1], X) \ \ \ \ & \text{ if $n = 0$;}\\
        \Gph([n],X) & \text{ if $n > 0$.}
      \end{cases}\\
      \Gph([n], QX) &=
      \begin{cases}
        \Gph([0],X) + 2 \cdot \Gph([1], X) & \text{ if $n = 0$;}\\
        \Gph([1], X) + \Gph([1], X) & \text{ if $n = 1$;}\\
        \Gph([n],X) & \text{ if $n > 1$.}
      \end{cases}
    \end{aligned}
  \end{equation}

  Now by definition, $N_K$ sends $K$-absolute colimits to colimits,
  whence also each $\Gph([n], \thg) \colon \Gph\rightarrow \cat{Set}$.
  The functors with this property are closed under colimits in
  $[\Gph, \Set]$, and so~\eqref{eq:36} ensures that each
  $\Gph([n], P\thg)$ and $\Gph([n], Q\thg)$ sends $K$-absolute colimits
  to colimits as desired.
\end{proof}

It is not even clear to us if the category of monads with arities $\A$
is always cocomplete. The argument for local presentability of $\nerv$
in Theorem~\ref{thm:7} does not seem to adapt to the case of monads
with arities, and no other obvious argument presents itself. In any
case, the preceding result shows that, even if the category of monads
with arities does have colimits, they do not always coincide with the
usual colimits of monads, and, in particular, are not always
\emph{algebraic}. This dashes any hope we might have had of giving a
sensible notion of presentation for monads with arities.

\section{Deferred proofs}
\label{sec:deferred-proofs}

\subsection{Identifying the monads}
\label{sec:proofs-from-section}
In this section, we complete the proofs of the results deferred from
Section~\ref{sec:identifying-theories} above, beginning with
Theorem~\ref{thm:7}. Recall that the category $\sig$ of
\emph{signatures} is the (ordinary) category
$\V\text-\cat{CAT}(\ob \A, \E)$, and that
$V \colon \mnd \rightarrow \sig$ is the functor sending $\mathsf{T}$
to $(Ta)_{a \in \A}$.

\begin{Prop}
  \label{prop:8}
  $V \colon \mnd \rightarrow \sig$ has a left adjoint $F$ which takes
  values in $\A$-nervous monads.
\end{Prop}

\begin{proof}
  We can decompose $V$ as the composite
  \begin{equation*}
    \mnd \xrightarrow{V_1} \V\text-\cat{CAT}(\E, \E)
    \xrightarrow{V_2} \sig
  \end{equation*}
  where $V_1$ takes the underlying endofunctor, and $V_2$ is given by
  evaluation at each $a \in \ob \A$. Since $V_2$ is equally given by
  restriction along $\ob \A \rightarrow \A \rightarrow \E$, it has a
  left adjoint $F_2$ given by pointwise left Kan extension, with the
  explicit formula:
  \begin{equation*}
    F_2(\Sigma) = \textstyle\sum_{a \in \A} \E(Ka, \thg) \cdot \Sigma a \colon \E
    \rightarrow \E\rlap{ ,}
  \end{equation*}
  where $\cdot$ denotes $\V$-enriched copower. So it suffices to show
  that the free monad on each endofunctor $F_2(\Sigma)$ exists and is
  $\A$-nervous. By~\cite[Theorem~23.2]{Kelly1980A-unified}, such a free
  monad $\mathsf{T}$ is characterised by the property that
  $\Alg{\mathnormal{F_2(\Sigma)}} \cong \Alg{T}$ over $\E$, where on
  the left we have the $\V$-category of algebras for the mere
  endofunctor $F_2(\Sigma)$. Thus, to complete the proof, it suffices
  by Theorem~\ref{thm:adjunction} to exhibit
  $\Alg{\mathnormal{F_2(\Sigma)}}$ as isomorphic to the $\V$-category
  of concrete models of some $\A$-pretheory.

  To this end, we let $\B$ be the \emph{collage} of the $\V$-functor
  $N_K \Sigma \colon \ob \A \rightarrow [\A^\mathrm{op}, \V]$. Thus
  $\B$ is the $\V$-category with object set $\ob \A + \ob \A$ and the
  following hom-objects, where we write
  $\ell,r \colon \ob \A \rightarrow \ob \B$ for the two injections:
  \begin{align*}
    \B(\ell a', \ell a) &= \A(a',a) & \B(r a', ra) &= (\ob \A)(a',a) \\ \B(\ell a', ra) &=
    \E(Ka', \Sigma a) & \B(ra', \ell a) &= 0\rlap{ .}
  \end{align*}
  Let $\ell \colon \A \rightarrow \B$ and
  $r \colon \ob \A \rightarrow \B$ be the two injections into the
  collage, and now form the pushout $J \colon \A \rightarrow \T$ of
  $\spn{\ell, r} \colon \A + \ob \A \rightarrow \B$ along
  $\spn{1, \iota} \colon \A + \ob \A \rightarrow \A$. Since
  $\spn{\ell, r}$ is identity-on-objects, so is
  $J \colon \A \rightarrow \T$, and so we have an $\A$-pretheory. To
  conclude the proof, it now suffices to show that
  $\E^{F_2(\Sigma)} \cong \cat{Mod}_c(\T)$ over $\E$.

  By the universal property of the collage and the pushout, to give a
  functor $H \colon \T \rightarrow \X$ is equally to give a functor
  $F = HJ \colon \A \rightarrow \X$ together with $\V$-natural
  transformations
  $\alpha_a \colon \E(K\thg, \Sigma a) \Rightarrow \X(F\thg, Fa)$ for
  each $a \in \ob \A$. In particular, taking $\X = \V^\mathrm{op}$ and
  $F = \E(K\thg, X)$, we see that a concrete $\T$-model structure on
  $X \in \E$ is given by an $\ob \A$-indexed family of $\V$-natural
  transformations
  \begin{equation*}
    \alpha_a \colon \E(K\thg, \Sigma a) \Rightarrow [\E(Ka, X),
    \E(K\thg, X)]
  \end{equation*}
  or equally under transpose, by a family of maps
  \begin{equation*}
    \E(Ka, X) \rightarrow [\A^\mathrm{op}, \V](\E(K\thg, \Sigma a),
    \E(K\thg, X))\rlap{ .}
  \end{equation*}
  By full fidelity of $N_K$, the right-hand side above is isomorphic to
  $\E(\Sigma a, X)$, and so concrete $\T$-model structure on $X$ is
  equally given by a family of maps
  $\E(Ka, X) \rightarrow \E(\Sigma a, X)$. Finally, using the universal
  properties of copowers and coproducts, this is equivalent to giving a
  single map
  \begin{equation*}
    \bar \alpha \colon \textstyle\sum_{a \in \A} \E(Ka, X) \cdot \Sigma a \rightarrow X
  \end{equation*}
  exhibiting $X$ as an $F_2(\Sigma)$-algebra. We thus have a bijection
  over $\E$ between objects of $\Alg{\mathnormal{F_2(\Sigma)}}$ and
  objects of $\conc$.

  A similar analysis shows that a morphism $A \rightarrow \E(X,Y)$ in
  $\V$ lifts through the monomorphism
  $\conc((X, \alpha), (Y, \beta)) \rightarrow \E(X,Y)$ if and only if
  it lifts through the monomorphism
  $\E^{F_2(\Sigma)}((X,\bar \alpha),(Y,\bar \beta)) \rightarrow
  \E(X,Y)$. It follows that we have an isomorphism of $\V$-categories
  $\Alg{\mathnormal{F_2(\Sigma)}} \cong \conc$ over $\E$ as desired.
\end{proof}

In proving the rest of Theorem~\ref{thm:7}, the following lemma will
be useful.

\begin{Lemma}
  \label{lem:7}
  Let $\C_1 \subseteq \C_2$ be replete, full, colimit-closed
  sub-$\V$-categories of $\C$; for example, they could be
  coreflective. If $V \colon \C \rightarrow \D$ has a left adjoint $F$
  taking values in $\C_1$, and the restriction
  $\res V {\C_2} \colon \C_2 \rightarrow \D$ is monadic, then
  $\C_1 = \C_2$.
\end{Lemma}

\begin{proof}
  Since $F$ takes values in $\C_1 \subseteq \C_2$, the left adjoint to
  $\res V{\C_2} \colon \C_1 \rightarrow \D$ is still given by $F$. So
  monadicity of $\res V {\C_2}$ means that each $X \in \C_2$ can be
  written as a coequaliser in $\C_2$, and hence also in $\C$, of
  objects in the image of $F$. Since $\im F \subseteq \C_1$ and since
  $\C_1$ is closed in $\C$ under colimits, it follows that
  $X \in \C_1$.
\end{proof}

\thmseven*

\begin{proof}
  We begin with (i). Let
  $H \colon \preth \rightarrow \V\text-\cat{CAT}(\ob \A,
  [\A^\mathrm{op}, \V])$ be the functor sending a pretheory
  $J \colon \A \rightarrow \T$ to the family of presheaves
  $(\T(J-,Ja))_{a \in \A}$. Since an $\A$-pretheory is a theory just
  when each of these presheaves is a $K$-nerve, we have a pullback
  square as to the right in:
  \begin{equation}\label{eq:17}\hskip-1em
    \cd[@C+0em]{
      {\nerv} \ar[r]^-{\Phi} \ar[d]_{V}
      \twocong{dr}{} &
      {\th} \pullbackcorner \ar[d]^{P} \ar@{ (->}[r] & \preth \ar[d]^{H} \\
      {\V\text-\cat{CAT}(\ob \A, \E)} \ar[r]_-{N_K \circ (\thg)} &
      {\V\text-\cat{CAT}(\ob \A, \KN)} \ar@{ (->}[r] & {\V\text-\cat{CAT}(\ob \A, [\A^\mathrm{op}, \V])}\rlap{ .}
    } 
  \end{equation}
  Since $\KN \hookrightarrow [\A^\mathrm{op}, \V]$ is replete, this
  square is a pullback along a discrete isofibration, and so
  by~\cite[Corollary~1]{Joyal1993Pullbacks} also a bipullback. On the
  other hand, to the left, we have a pseudocommuting square as
  witnessed by the isomorphisms:
  \begin{equation*}
    (PJ_\mathsf{T})(A) = \A_\mathsf{T}(J_\mathsf{T} \thg, J_\mathsf{T}A) = \E^\mathsf{T}(F^\mathsf{T}K\thg,
    F^\mathsf{T}KA) \cong \E(K\thg, TKA) = N_K(TKA)\rlap{ .}
  \end{equation*}
  Since both horizontal edges of this square are equivalences, it is
  also a bipullback.

  To show the required monadicity, we must prove that $V$ creates
  $V$-absolute coequalisers. Since the large rectangle is a
  bipullback---as the pasting of two bipullbacks---it suffices to show
  that $H$ creates $H$-absolute coequalizers. As the definition of $H$
  depends only on $\A$ and not $\E$, we lose no generality in proving
  this if we assume that $\E = [\A^\mathrm{op}, \V]$ and $K = Y$. In
  this case, \emph{every} presheaf on $\A$ is a $K$-nerve, and so the
  horizontal composites in~\eqref{eq:17} are equivalences; and so,
  finally, it suffices to prove that $V$ is monadic when
  $\E = [\A^\mathrm{op}, \V]$ and $K = Y$.

  Note that, in this case, $\A$ is a saturated class of arities: for
  indeed, by the universal property of free cocompletion, a functor
  $F \colon [\A^\mathrm{op}, \V] \rightarrow [\A^\mathrm{op}, \V]$ is
  $\A$-\induced if and only if it is \emph{cocontinuous}. It thus
  follows from Proposition~\ref{prop:3} below that the restriction
  $V_c \colon \cat{Mnd}_c(\E) \rightarrow \sig$ of $V$ to cocontinuous
  monads is monadic; so we will be done if $\cat{Mnd}_c(\E) = \nerv$.
  In this case, $\Psi \colon \preth \rightarrow \nerv$ sends
  $J \colon \A \rightarrow \T$ to a monad which is isomorphic to that
  induced by the adjunction
  $\mathrm{Lan}_J \colon [\A^\mathrm{op}, \V] \leftrightarrows
  [\T^\mathrm{op}, \V] \colon [J^\mathrm{op}, 1]$, and so
  $\nerv \subseteq \cat{Mnd}_c(\E)$. To obtain equality, we apply
  Lemma~\ref{lem:7}. We have that:
  \begin{itemize}[itemsep=0em]
  \item $\nerv$ and $\cat{Mnd}_c(\E)$ are coreflective in $\mnd$ by
    Corollary~\ref{cor:1} and Lemma~\ref{lem:8} respectively;
  \item $V \colon \mnd \rightarrow \sig$ has a left adjoint taking
    values in $\nerv$;
  \item The restriction $V_c \colon \cat{Mnd}_c(\E) \rightarrow \sig$
    is monadic;
  \end{itemize}
  and so $\nerv = \cat{Mnd}_c(\E)$. This proves monadicity of $V$ in
  the special case $\E = [\A^\mathrm{op}, \V]$, whence also, by the
  preceding argument, in the general case.

  In order to prove (ii), we let $\C_1$ be the colimit-closure in
  $\mnd$ of the image of $F$. Since $\nerv$ contains this image and is
  colimit-closed, we have $\C_1 \subseteq \nerv \subseteq \mnd$. Thus,
  applying Lemma~\ref{lem:7} to this triple and
  $V \colon \nerv \rightarrow \sig$ gives $\nerv = \C_1$ as desired.

  Finally we prove (iii). The monadicity of $V$ above implies that of
  $P$ and hence also of $H$ (by taking $\E = [\A^\mathrm{op}, \V])$.
  Since filtered colimits of $\A$-pretheories can be computed at the
  level of underlying \emph{graphs}, the forgetful $H$ preserves them;
  which is to say that $\preth$ is \emph{finitarily} monadic over the
  locally presentable
  $\V\text-\cat{CAT}(\ob \A, [\A^\mathrm{op}, \V])$, whence locally
  presentable by~\cite[Satz~10.3]{Gabriel1971Lokal}. So in the
  right-hand and the large bipullback squares in~\eqref{eq:17}, the
  bottom and right sides are right adjoints between locally presentable
  categories. Since by~\cite[Theorem~2.18]{Bird1984Limits}, the
  $2$-category of locally presentable categories and right adjoint
  functors is closed under bilimits in $\CAT$, we conclude that each
  $\th$ and each $\nerv$ is also locally presentable.
\end{proof}

\subsection{Saturated classes}
\label{sec:saturated-classes}

We now turn to the deferred proof of Theorem~\ref{thm:5}. Recall the
context: an endo-$\V$-functor $F \colon \E \rightarrow \E$ is called
\emph{$\A$-\induced} when the pointwise left Kan extension of its
restriction along $K$, and $\A$ is a \emph{saturated class of arities}
if $\A$-\induced endofunctors of $\E$ are composition-closed.

We begin by recording the basic properties of this situation. We write
$\aend$ and $\amnd$ for the full subcategories of
$\eend = \V\text-\cat{CAT}(\E, \E)$ and $\mnd$ on, respectively, the
$\A$-\induced endofunctors, and the monads with $\A$-\induced
underlying endofunctor.

\begin{Lemma}
  \label{lem:8}
  $\aend$ is coreflective in $\eend = \V\text-\cat{CAT}(\E, \E)$ via
  the coreflector $R(F) = \mathrm{Lan}_K(FK)$, as on the left in:
  \begin{equation}
    \label{eq:32}
    \cd{
      {\aend} \ar@<-4.5pt>[r]_-{I} \ar@{}[r]|-{\top} &
      {\eend} \ar@<-4.5pt>[l]_-{R} & &
      {\amnd} \ar@<-4.5pt>[r]_-{I} \ar@{}[r]|-{\top} &
      {\mnd} \ar@<-4.5pt>[l]_-{R}\rlap{ .}
    }
  \end{equation}
  If $\A$ is a saturated class, then $\aend$ is right-closed monoidal,
  and the coreflection left above lifts to the corresponding categories
  of monads as on the right.
\end{Lemma}
\begin{proof}
  Restriction and left Kan extension along the fully faithful $K$
  exhibits $\aend$ as equivalent to $\V\text-\cat{CAT}(\A, \E)$,
  whence locally presentable. Since restriction along $K$ is a
  coreflector of $\eend$ into $\V\text-\cat{CAT}(\A, \E)$, it follows
  that $R(F) = \mathrm{Lan}_K(FK)$ is a coreflector of $\eend$ into
  $\aend$.

  If $\A$ is saturated then $\aend$ is monoidal under composition.
  Since each endofunctor $(\thg) \circ F$ of $\eend$ is cocontinuous,
  and $\aend$ is closed in $\eend$ under colimits, each endofunctor
  $(\thg) \circ F$ of $\aend$ is cocontinuous, and so has a right
  adjoint by local presentability. Thus $\aend$ is right-closed
  monoidal.

  Furthermore, the inclusion of $\aend$ into $\eend$ is strict
  monoidal, whence by~\cite[Theorem~1.5]{Kelly1974Doctrinal} the
  coreflection to the left of~\eqref{eq:32} lifts to a coreflection in
  the $2$-category $\cat{MONCAT}$ of monoidal categories, lax monoidal
  functors and monoidal transformations. Applying the $2$-functor
  $\cat{MONCAT}(1, \thg) \colon \cat{MONCAT} \rightarrow \cat{CAT}$
  yields the coreflection to the right of~\eqref{eq:32}.
\end{proof}

The key step towards establishing Theorem~\ref{thm:5} above is now:

\begin{Prop}
  \label{prop:3}
  The left adjoint $F$ of $V \colon \mnd \rightarrow \sig$ takes
  values in $\A$-\induced monads; furthermore, the restriction of $V$
  to $\amnd$ is monadic.
\end{Prop}
\begin{proof}
  For any $\mathsf{T} \in \mnd$, its $\A$-\induced coreflection
  $\varepsilon_\mathsf{T} \colon IR(\mathsf{T}) \rightarrow
  \mathsf{T}$ has as underlying map in $\eend$ the component
  $\mathrm{Lan}_K(TK) \rightarrow T$ of the counit of the adjunction
  given by restriction and left Kan extension along $K$. Since $K$ is
  fully faithful, the restriction of this map along $K$ is invertible,
  whence in particular,
  $V \varepsilon \colon VIR \Rightarrow V \colon \mnd \rightarrow
  \sig$ is invertible. So $\eta \colon \id \Rightarrow VF$ factors
  through $V \varepsilon_F \colon VIRF \Rightarrow VF$ whence, by
  adjointness, $\id \colon F \Rightarrow F$ factors through
  $\varepsilon_F$. Therefore each $F(\Sigma)$ is a retract of
  $IRF(\Sigma)$; since $\amnd$ is closed under colimits in $\mnd$, it
  is retract-closed and so each $F(\Sigma)$ belongs to $\amnd$.

  It remains to prove that the restriction of $V$ to $\amnd$ is
  monadic. To do so, we decompose this restriction as
  \begin{equation*}
    \A\text-\mnd \xrightarrow{V_1}
    \A\text-\cat{End}(\E) \xrightarrow{V_2} \sig\rlap{ ,}
  \end{equation*}
  where $V_1$ forgets the monad structure and $V_2$ is given by
  precomposition with $\ob \A \rightarrow \A \to \E$, and apply the
  following result, which
  is~\cite[Theorem~2]{Lack1999On-the-monadicity}:

  \begin{Thm*}
    Let $\M$ be a right-closed monoidal category, and
    $V_2 \colon \M \rightarrow \N$ a monadic functor for which there
    exists a functor
    $\mathord\diamond \colon \M \times \N \rightarrow \N$ with natural
    isomorphisms $X \diamond VY \cong V(X \otimes Y)$. If the forgetful
    functor $V_1 \colon \cat{Mon}(\M) \rightarrow \M$ has a left
    adjoint, then the composite
    $V_2V_1 \colon \cat{Mon}(\M) \rightarrow \N$ is monadic.
  \end{Thm*}

  Indeed, by Lemma~\ref{lem:8}, $\aend$ is a right-closed monoidal
  category, and $\amnd$ the category of monoids therein. Under the
  equivalence $\aend \simeq \V\text-\cat{CAT}(\A, \E)$, we may identify
  $V_2$ with precomposition along $\ob \A \rightarrow \A$. It is thus
  cocontinuous, and has a left adjoint given by left Kan extension;
  whence is monadic. Now since $V_2V_1$ has a left adjoint and $V_2$ is
  monadic, it follows that $V_1$ also has a left adjoint. Finally, we
  have a functor
  \begin{equation*}
    \diamond \colon \aend \times
    \sig \rightarrow
    \sig
  \end{equation*}
  defined by $(F,G) \mapsto FG$, and this clearly has the property that
  $M(FG) = F \diamond M(G)$. So applying the above theorem yields the
  desired monadicity.
\end{proof}

We are now ready to prove:

\thmfive*

\begin{proof}
  For (i) $\Leftrightarrow$ (ii), the monadicity of
  $V \colon \amnd \rightarrow \sig$ verified in the previous
  proposition implies, as in the proof of Theorem~\ref{thm:7}(iii),
  that $\amnd$ is the colimit-closure in $\mnd$ of the free monads on
  signatures. Since $\nerv$ is also this closure, we have
  $\nerv = \amnd$ as desired. For (ii)~$\Leftrightarrow$~(iii), we
  apply Proposition~\ref{prop:7}.
\end{proof}

\bibliographystyle{acm}
\bibliography{bibdata}

\begin{thebibliography}{10}

\bibitem{Adamek1994Locally}
{\sc Ad{\'a}mek, J., and Rosick{\'y}, J.}
\newblock {\em Locally presentable and accessible categories}, vol.~189 of {\em
  London Mathematical Society Lecture Note Series}.
\newblock Cambridge University Press, 1994.

\bibitem{Ara2010Sur-les-infty-groupoides}
{\sc Ara, D.}
\newblock {\em Sur les {$\infty$}-groupo{\"\i}des de {G}rothendieck et une
  variante {$\infty$}-cat{\'e}gorique}.
\newblock PhD thesis, Universit{\'e} Paris VII, 2010.

\bibitem{Ara2013On-the-homotopy}
{\sc Ara, D.}
\newblock On the homotopy theory of {G}rothendieck {$\infty$}-groupoids.
\newblock {\em Journal of Pure and Applied Algebra 217\/} (2013), 1237--1278.

\bibitem{Avery2017Structure}
{\sc Avery, T.}
\newblock {\em Structure and semantics}.
\newblock PhD thesis, University of Edinburgh, 2017.

\bibitem{Barr1970Coequalizers}
{\sc Barr, M.}
\newblock Coequalizers and free triples.
\newblock {\em Mathematische Zeitschrift 116\/} (1970), 307--322.

\bibitem{Barr1985Toposes}
{\sc Barr, M., and Wells, C.}
\newblock {\em Toposes, triples and theories}, vol.~278 of {\em Grundlehren der
  Mathematischen Wissenschaften}.
\newblock Springer, 1985.

\bibitem{Batanin1998Monoidal}
{\sc Batanin, M.}
\newblock Monoidal globular categories as a natural environment for the theory
  of weak {$n$}-categories.
\newblock {\em Advances in Mathematics 136\/} (1998), 39--103.

\bibitem{Berger2002A-Cellular}
{\sc Berger, C.}
\newblock A cellular nerve for higher categories.
\newblock {\em Advances in Mathematics 169\/} (2002), 118--175.

\bibitem{Berger2012Monads}
{\sc Berger, C., Melli{\`e}s, P.-A., and Weber, M.}
\newblock Monads with arities and their associated theories.
\newblock {\em Journal of Pure and Applied Algebra 216\/} (2012), 2029--2048.

\bibitem{Bird1984Limits}
{\sc Bird, G.}
\newblock {\em Limits in 2-categories of locally-presented categories}.
\newblock PhD thesis, University of Sydney, 1984.

\bibitem{Day1974On-closed}
{\sc Day, B.}
\newblock On closed categories of functors {II}.
\newblock In {\em Category Seminar (Sydney, 1972/1973)}, vol.~420 of {\em
  Lecture Notes in Mathematics}. Springer, 1974, pp.~20--54.

\bibitem{Dubuc1970Kan-extensions}
{\sc Dubuc, E.~J.}
\newblock {\em Kan extensions in enriched category theory}, vol.~145 of {\em
  Lecture Notes in Mathematics}.
\newblock Springer, 1970.

\bibitem{Gabriel1971Lokal}
{\sc Gabriel, P., and Ulmer, F.}
\newblock {\em Lokal pr{\"a}sentierbare {K}ategorien}, vol.~221 of {\em Lecture
  Notes in Mathematics}.
\newblock Springer, 1971.

\bibitem{Joyal1993Pullbacks}
{\sc Joyal, A., and Street, R.}
\newblock Pullbacks equivalent to pseudopullbacks.
\newblock {\em Cahiers de Topologie et Geom{\'e}trie Diff{\'e}rentielle
  Cat{\'e}goriques 34\/} (1993), 153--156.

\bibitem{Kelly1974Doctrinal}
{\sc Kelly, G.~M.}
\newblock Doctrinal adjunction.
\newblock In {\em Category Seminar (Sydney, 1972/1973)}, vol.~420 of {\em
  Lecture Notes in Mathematics}. Springer, 1974, pp.~257--280.

\bibitem{Kelly1980A-unified}
{\sc Kelly, G.~M.}
\newblock A unified treatment of transfinite constructions for free algebras,
  free monoids, colimits, associated sheaves, and so on.
\newblock {\em Bulletin of the Australian Mathematical Society 22\/} (1980),
  1--83.

\bibitem{Kelly1982Basic}
{\sc Kelly, G.~M.}
\newblock {\em Basic concepts of enriched category theory}, vol.~64 of {\em
  London Mathematical Society Lecture Note Series}.
\newblock Cambridge University Press, 1982.
\newblock Republished as: \textit{Reprints in Theory and Applications of
  Categories 10} (2005).

\bibitem{Kelly1982Structures}
{\sc Kelly, G.~M.}
\newblock Structures defined by finite limits in the enriched context {I}.
\newblock {\em Cahiers de Topologie et Geom\'{e}trie Diff\'{e}rentielle
  Cat\'{e}goriques 23\/} (1982), 3--42.

\bibitem{Kelly1993Finite-product-preserving}
{\sc Kelly, G.~M., and Lack, S.}
\newblock Finite-product-preserving functors, {K}an extensions and
  strongly-finitary {$2$}-monads.
\newblock {\em Applied Categorical Structures 1\/} (1993), 85--94.

\bibitem{Kelly1993Adjunctions}
{\sc Kelly, G.~M., and Power, A.~J.}
\newblock Adjunctions whose counits are coequalizers, and presentations of
  finitary enriched monads.
\newblock {\em Journal of Pure and Applied Algebra 89\/} (1993), 163--179.

\bibitem{Lack1999On-the-monadicity}
{\sc Lack, S.}
\newblock On the monadicity of finitary monads.
\newblock {\em Journal of Pure and Applied Algebra 140\/} (1999), 65--73.

\bibitem{Lack2009Gabriel-Ulmer}
{\sc Lack, S., and Power, J.}
\newblock Gabriel-{U}lmer duality and {L}awvere theories enriched over a
  general base.
\newblock {\em Journal of Functional Programming 19\/} (2009), 265--286.

\bibitem{Lack2011Notions}
{\sc Lack, S., and Rosick{\'y}, J.}
\newblock Notions of {L}awvere theory.
\newblock {\em Applied Categorical Structures 19\/} (2011), 363--391.

\bibitem{Lawvere1963Functorial}
{\sc Lawvere, F.~W.}
\newblock {\em Functorial semantics of algebraic theories}.
\newblock PhD thesis, Columbia University, 1963.
\newblock Also \emph{Proc. Nat. Acad. Sci. U.S.A. 50} (1963), 869--872.
  Republished as: \textit{Reprints in Theory and Applications of Categories 5}
  (2004).

\bibitem{Leinster2004Operads}
{\sc Leinster, T.}
\newblock Operads in higher-dimensional category theory.
\newblock {\em Theory and Applications of Categories 12\/} (2004), 73--194.

\bibitem{Linton1966Some}
{\sc Linton, F. E.~J.}
\newblock Some aspects of equational categories.
\newblock In {\em Conference on Categorical Algebra (La Jolla, 1965)}.
  Springer, 1966, pp.~84--94.

\bibitem{Maltsiniotis2010Grothendieck}
{\sc Maltsiniotis, G.}
\newblock Grothendieck {$\infty$}-groupoids, and still another definition of
  {$\infty$}-categories.
\newblock Unpublished, available as \url{https://arxiv.org/abs/1009.2331},
  2010.

\bibitem{Markowsky1976Chain-complete}
{\sc Markowsky, G.}
\newblock Chain-complete posets and directed sets with applications.
\newblock {\em Algebra Universalis 6}, 1 (1976), 53--68.

\bibitem{Mellies2010Segal}
{\sc Melli{{\`e}}s, P.-A.}
\newblock Segal condition meets computational effects.
\newblock In {\em 25th {A}nnual {IEEE} {S}ymposium on {L}ogic in {C}omputer
  {S}cience {LICS} 2010}. IEEE Computer Society Press, 2010, pp.~150--159.

\bibitem{Meyer1975Induced}
{\sc Meyer, J.-P.}
\newblock Induced functors on categories of algebras.
\newblock {\em Mathematische Zeitschrift 142\/} (1975), 1--14.

\bibitem{Nishizawa2009Lawvere}
{\sc Nishizawa, K., and Power, J.}
\newblock Lawvere theories enriched over a general base.
\newblock {\em Journal of Pure and Applied Algebra 213\/} (2009), 377--386.

\bibitem{Power1999Enriched}
{\sc Power, J.}
\newblock Enriched {L}awvere theories.
\newblock {\em Theory and Applications of Categories 6\/} (1999), 83--93.

\bibitem{Segal1968Classifying}
{\sc Segal, G.}
\newblock Classifying spaces and spectral sequences.
\newblock {\em Institut des Hautes {\'E}tudes Scientifiques. Publications
  Math{\'e}matiques 34\/} (1968), 105--112.

\bibitem{Street2000The-petit}
{\sc Street, R.}
\newblock The petit topos of globular sets.
\newblock {\em Journal of Pure and Applied Algebra 154\/} (2000), 299--315.

\bibitem{Weber2007Familial}
{\sc Weber, M.}
\newblock Familial 2-functors and parametric right adjoints.
\newblock {\em Theory and Applications of Categories 18\/} (2007), 665--732.

\end{thebibliography}

\end{document}